\documentclass[12pt]{amsart}

\usepackage{amsthm,latexsym,amssymb,amsfonts,amsmath,url,mathtools,mathrsfs}
	
\usepackage{placeins}
\usepackage{xcolor}
\usepackage[vcentermath,enableskew]{youngtab}
\usepackage{color}
\usepackage[noadjust]{cite}
\usepackage[margin=1in]{geometry}
\usepackage{enumitem, bm}
\usepackage[unicode,psdextra]{hyperref}
\usepackage{tikz-cd}

\usepackage{ytableau}

\usepackage{graphicx}
\graphicspath{{./Images/}}

\usepackage{tikz}
\usetikzlibrary{decorations.markings}

\usetikzlibrary{calc, shapes, backgrounds,arrows,positioning,plotmarks}
\tikzset{>=stealth',
	head/.style = {fill = white, text=black},
	plaque/.style = {draw, rectangle, minimum size = 10mm}, 
	pil/.style={->,thick},
	junct/.style = {draw,circle,inner sep=0.5pt,outer sep=0pt, fill=black}
}

\newtheorem{theorem}{Theorem}[section]
\newtheorem{proposition}[theorem]{Proposition}
\newtheorem{lemma}[theorem]{Lemma}
\newtheorem{conjecture}[theorem]{Conjecture}
\newtheorem{corollary}[theorem]{Corollary}

\theoremstyle{definition}
\newtheorem{definition}[theorem]{Definition}
\newtheorem{remark}[theorem]{Remark}
\newtheorem{setup}[theorem]{Setup}
\newtheorem{examplex}[theorem]{Example}
\newenvironment{example}
{\pushQED{\qed}\examplex}
{\popQED\endexamplex}

\newtheoremstyle{named}{}{}{\itshape}{}{\bfseries}{.}{.5em}{#1 \thmnote{#3}}
\theoremstyle{named}

\numberwithin{equation}{section}

\newcommand{\wt}{\mathrm{wt}}
\newcommand{\Koh}{{\sf Koh}}
\newcommand{\KKoh}{{\sf KKoh}}
\newcommand{\Z}{\mathbb{Z}}
\newcommand{\newword}[1]{\emph{#1}}
\newcommand{\PD}{\mathrm{PD}}
\newcommand{\RPD}{\mathrm{RPD}}
\newcommand{\GKKoh}{{\sf GKKoh}}
\newcommand{\SVRT}{\operatorname{SVRT}}
\newcommand{\SRT}{\operatorname{SRT}}

\newcommand{\K}{\ensuremath{K}}
\newcommand{\SKoh}{{\sf SKoh}}
\newcommand{\SKKoh}{{\sf SKKoh}}

\title[]{Counterexamples to Robichaux's conjecture for Grothendieck polynomials}

\author{Avery St.~Dizier}
\address[AS]{
	Department of Mathematics,
	Michigan State University, East Lansing,
	MI 48824, USA}
\email{stdizier@msu.edu}

\thanks{}

\begin{document}
	\begin{abstract}
		Ross and Yong conjectured a \K-theoretic Kohnert rule for Grothendieck polynomials. Robichaux exhibited a counterexample to the Ross--Yong rule and proposed a revised ghost \K-Kohnert rule, proving both rules hold for 321-avoiding permutations. We provide counterexamples to Robichaux's rule and give an explicit bijection showing that both the Ross--Yong and Robichaux rules hold for 1432-avoiding permutations. As an application, we provide a Kohnert-theoretic characterization of 1432-avoidance.
	\end{abstract}

	\maketitle
	
	\section{Introduction}
	\label{sec:intro}
	
	Schubert polynomials, multivariate polynomials indexed by permutations, were introduced by Lascoux and Sch\"utzenberger \cite{LS1} as polynomial representatives for Schubert classes in the cohomology of the flag variety. Kohnert conjectured a positive diagrammatic rule for Schubert polynomials, constructing a family of diagrams recursively by iterating simple moves \cite{original_kohnert}. A similar recursive diagrammatic rule (with different moves) was then proved by Bergeron \cite{bergeron_moves}. Many combinatorial models for Schubert polynomials have since been found \cite{BJS,nilcoxeter,FKschub,prismtableaux,laddermoves,magyar,bumpless_pipe_dreams,unifiedschubert,balancedtableaux}.
	
	While Bergeron sketched a possible argument that his and Kohnert's rules construct the same family of diagrams \cite{macdonald}, the first full proof of Kohnert's rule was given by Winkel \cite{winkel1}, who later produced a second proof \cite{winkel2}. Motivated by the complexities present in both of Winkel's proofs, Assaf gave a short bijective proof in \cite{assaf_kohnert}. An additional proof was given by Armon, Assaf, Bowling, and Ehrhard in \cite{assaf_flagged_schur}.
	
	Our interest in the present paper is a \K-theoretic generalization of Kohnert's rule. In \K-theory, Grothendieck polynomials represent the classes of structure sheaves of Schubert varieties \cite{LS2}. In \cite{ghost_kohnert}, Ross and Yong gave a conjectured \K-theoretic Kohnert rule. They conjectured that their rule provided a combinatorial model for both Grothendieck polynomials and Lascoux polynomials (the \K-theoretic Demazure characters). Building on work of Pechenik and Scrimshaw on rectangular tableaux \cite{rectangle_k_crystal}, Pan and Yu proved the Ross--Yong rule for Lascoux polynomials via a recursive bijection \cite{lascoux_k_kohnert}. Utilizing the Pan--Yu bijection, Robichaux then disproved the Ross--Yong conjecture, proposed a revised ghost \K-Kohnert rule, and verified that both the Robichaux and Ross--Yong rules hold for 321-avoiding permutations \cite{ghost_kohnert_fix}. 
	
	Our first main result is that Robichaux's conjectured rule is also false. We demonstrate that it can both overcount (Proposition~\ref{prop:ghost_koh_overcounts}) and undercount (Proposition~\ref{prop:undercount_proof}) the coefficients of Grothendieck polynomials. We conjecture that Robichaux's rule correctly computes the support of Grothendieck polynomials (Conjecture~\ref{conj:support}).

	Our second main result is a generalization of Robichaux's 321-avoiding positive result \cite[Theorem 5.1]{ghost_kohnert_fix}. We prove that both the Robichaux and Ross--Yong rules hold for the class of 1432-avoiding permutations. This is accomplished via an explicit bijection to set-valued tableaux of Fan and Guo \cite{fg_svrt} in Theorem~\ref{thm:bijection}. As an application of the bijection, we give a \K-Kohnert-theoretic characterization of 1432-avoiding permutations (Theorem~\ref{thm:1432_kohnert_characterization}).
	
\subsection*{Outline} The paper is organized as follows. In Section~\ref{sec:back}, we review all necessary background for the paper. Then in Section~\ref{sec:conjectures}, we recall the Ross--Yong and Robichaux \K-Kohnert conjectures. We examine counterexamples to both. We construct in Section~\ref{sec:1432} a bijection proving the Ross--Yong conjectures for 1432-avoiding permutations. We then deduce the and Robichaux conjecture as a consequence. As an application of the bijection, we prove in Section~\ref{sec:simple} a new characterization of 1432-avoiding permutations analogous to Gao's \cite[Theorem 4.1]{gao_prin_spec}. Appendix~\ref{sec:appendix} gives explicit move sequences for witness ghost \K-Kohnert diagrams appearing in Section~\ref{sec:conjectures}.

	\section{Background}
	\label{sec:back}
	
	We provide all background on Schubert and Grothendieck polynomials, pipe dreams, and Kohnert diagrams used throughout the paper. Readers familiar with these objects may skip this section and return as needed.
	
	\subsection{Permutations and Patterns}\phantom{}\newline\vspace{-2ex}
	
	Let $S_n$ denote the symmetric group on the set $[n]\coloneqq\{1,2,\ldots,n\}$. We write permutations $w\in S_n$ in one-line notation as words $w(1)w(2)\cdots w(n)$. For example, $w=12543$ is the transposition usually denoted $(35)$, the bijection $[5]\to [5]$ taking $1\mapsto 1$, $2\mapsto 2$, $3\mapsto 5$, $4\mapsto 4$, and $5\mapsto 3$. Write $s_i$ for the transposition swapping $i$ and $i+1$. Let $w_0$ be the longest permutation $n(n-1)\cdots 21$. Denote by $\ell(w)$ the Coxeter length of $w$.
	
	We say $w\in S_n$ contains a pattern $\sigma\in S_k$ (for $k\leq n$) if there are indices $1\leq j_1<j_2<\cdots<j_k\leq n$ such that the numbers $w(j_1),\cdots, w(j_k)$ are in the same relative order as $\sigma(1),\ldots,\sigma(k)$.
	We say that $w$ avoids the pattern $\sigma$ if $w$ does not contain $\sigma$. For example, $w = 154623$ contains the
	pattern $132$, but avoids the pattern $13245$.
	
	\subsection{Schubert and Grothendieck Polynomials}\phantom{}\newline\vspace{-2ex}
	
	Fix any $n\geq 1$ and consider the polynomial ring $\Z[x_1,\ldots,x_n]$. For $1\leq j <n$, let $\partial_j,\overline{\partial}_j$ be the operators on $\Z[x_1,\ldots,x_n]$ defined by
	\[\partial_j(f)=\frac{f-(s_j\cdot f)}{x_j-x_{j+1}}\qquad\mbox{and}\qquad \overline{\partial}_j(f)=\partial_j((1-x_{j+1})f).\]
	Indexed by permutations $w\in S_n$, the \newword{Schubert polynomials} $\mathfrak{S}_w$ and the \newword{Grothendieck polynomials} $\mathfrak{G}_w$ are defined recursively by 
	\begin{align*}
		\mathfrak{S}_w=\begin{cases}
			x_1^{n-1}x_2^{n-2}\cdots x_{n-1}&\mbox{ if } w=w_0,\\
			\partial_j \mathfrak{S}_{ws_j} &\mbox{ if } w(j)<w(j+1);
		\end{cases}
	\end{align*}
	\begin{align*}
		\mathfrak{G}_w=\begin{cases}
			x_1^{n-1}x_2^{n-2}\cdots x_{n-1}&\mbox{ if } w=w_0,\\
			\overline{\partial}_j \mathfrak{G}_{ws_j} &\mbox{ if } w(j)<w(j+1).
		\end{cases}
	\end{align*}
	
	We now recall the pipe dream model for Schubert and Grothendieck polynomials \cite{FKschub,laddermoves,FKgroth}. 
	A \newword{pipe dream} is a filling of the triangular grid $\{(i,j)\in[n]\times[n]\colon i + j \leq n\}$ with crossing tiles \includegraphics[scale=0.5]{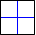} and bump tiles \includegraphics[scale=0.5]{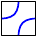}. By requiring half-bump tiles \includegraphics[scale=0.5]{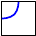} on the antidiagonal $\{(i,j)\in[n]\times[n]\colon i+j = n+1\}$, a pipe dream forms a network of $n$ pipes running from the top edge of the grid to the left edge (see Figure~\ref{fig:1423-pd}). 
	
	Associated to any pipe dream $P$ is a permutation $w\in S_n$ obtained as follows. Label the pipes $1$ through $n$ along the top edge. Trace the pipes downwards and leftwards, carrying their label along. Ignore any crossings between any pair of pipes which have already crossed, i.e.\ treat subsequent crossing tiles as if they were bump tiles. Read the labels of the pipes along the left edge from top to bottom to get the associated permutation $w$. A pipe dream is called \newword{reduced} if no pipes cross twice. 
	
	Write $\PD(w)$ for the set of pipe dreams of $w$, and $\RPD(w)$ for the set of reduced pipe dreams of $w$. The \newword{weight} of a pipe dream $P$ is the vector $\wt(P)$ with $i$\textsuperscript{th} component $\wt(P)_i=\mbox{\# crosses in row $i$ of }P$. For $\alpha\in \Z_{\geq 0}^n$, write 
	\begin{align*}
		\RPD(w,\alpha) &=\{P\in\RPD(w)\mid \wt(P)=\alpha\},\quad\mbox{and}\\
		\PD(w,\alpha) &=\{P\in\PD(w)\mid \wt(P)=\alpha\}.
	\end{align*}
	
	\begin{example}
		\label{exp:1423-pd}
		For $w=1423$, the set $\PD(w)$ consists of the five pipe dreams displayed in Figure~\ref{fig:1423-pd}. Note that bump and half-bump tiles are depicted in blue, first crossings in red, and redundant crossings in purple.
	\end{example}
	
	\begin{figure}[ht]
		\centering
		\includegraphics[scale=1]{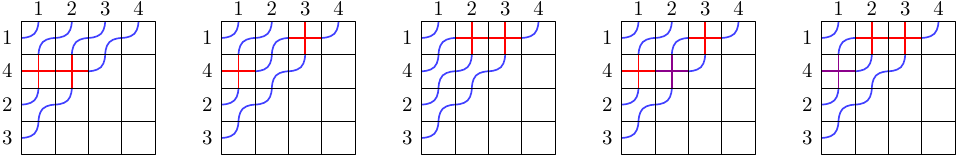}
		\caption{The five pipe dreams in $\PD(1423)$.}
		\label{fig:1423-pd}
	\end{figure}
	
	\begin{theorem}[{\cite{FKschub,laddermoves,FKgroth}}]
		\label{thm:pdformula}
		For any permutation $w\in S_n$,
		\[\mathfrak{S}_w=\sum_{P\in \RPD(w)} x^{\wt(P)}\quad \mbox{and} \quad \mathfrak{G}_w=\sum_{P\in \PD(w)} (-1)^{|\wt(P)|-\ell(w)}x^{\wt(P)}.\]
	\end{theorem}	
	In other words, for any $\alpha\in \Z_{\geq 0}^n$ one has
	\[[x^\alpha]\mathfrak{S}_w = \#\RPD(w,\alpha) \quad \mbox{and} \quad [x^{\alpha}]\mathfrak{G}_w = (-1)^{|\alpha|-\ell(w)} \#\PD(w,\alpha). \]
	
	We will be interested in proposed generalizations of the following family of objects also enumerating $[x^\alpha]\mathfrak{S}_w$.
	
	\subsection{Kohnert Diagrams}\phantom{}\newline\vspace{-2ex}
	
	A \newword{diagram} is a subset $D\subseteq [n]\times [n]$ of the $n\times n$ grid. We index the cells of the grid using matrix notation, so $(1,1)$ refers to the leftmost and topmost cell. We call the elements of a diagram \newword{boxes}. The \newword{weight} of a diagram $D$ is the vector $\wt(D)$ with $i$\textsuperscript{th} component 
	\[\wt(D)_i = \mbox{\# boxes in row $i$ of }D. \]
	
	The \newword{Rothe diagram} of $w\in S_n$ is the diagram 
	\[ D(w)=\{(i,j)\in [n]^2 \mid i<w^{-1}(j)\mbox{ and } j<w(i) \}. \]
	Note the Rothe cells $(i,j)\in D(w)$ are in bijection with the inversions of $w$ via $(i,j)\mapsto (i,w^{-1}(j))$, so $\#D(w)=\ell(w)$.
	
	The Rothe diagram $D(w)$ can be visualized as the set of cells left in the $n\times n$ grid 
	after you cross out all cells weakly below $(i,w(i))$ in the same column, or weakly right of $(i,w(i))$ in the same row for each $i\in [n]$. 
	
	\begin{example}
		For $w=31542$, one has
		\begin{center}
			\includegraphics[scale=1]{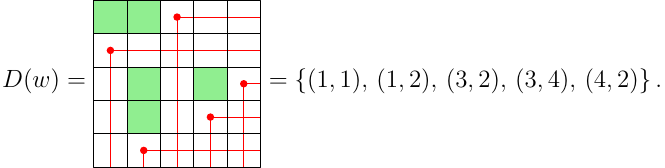}
		\end{center}
		We indicate surviving cells in green, and the removals with red lines.
	\end{example}
	
	Note that Rothe diagrams are always \newword{northwest diagrams}: if $(r, c'),(r', c)\in D(w)$ with $r < r'$ and $c < c'$, then $(r, c) \in D(w)$.
	
	Additionally, the boxes of $D(w)$ are in bijection with the inversions of $w$, so $\#D(w) = \ell(w)$.
	
	\begin{definition}
		For a diagram $D$ and $(i,j)\in D$ rightmost in its row, the \newword{Kohnert move} on $D$ at row $i$ produces the diagram
		\[D' = D\setminus\{(i,j)\} \cup \{(i',j)\} \]
		where $i'=\max\{r\in [i-1]\mid (r,j)\notin D \}$. (The move is not defined when the set is empty.)
	\end{definition}
	
	For a permutation $w$, write $\Koh(w)$ to denote the set of all diagrams obtainable from $D(w)$ by applying successive Kohnert moves. Set $\Koh(w,\alpha)=\{E\in\Koh(w)\mid \wt(E)=\alpha\}$. We will refer to the elements of $\Koh(w)$ as \newword{Kohnert diagrams}.
	
	\begin{example}
		The permutation $w=31542$ has the eight Kohnert diagrams shown in Figure~\ref{fig:kohnert-31542}.
	\end{example}
	
	\begin{figure}[ht]
		\centering
		\includegraphics[scale=1]{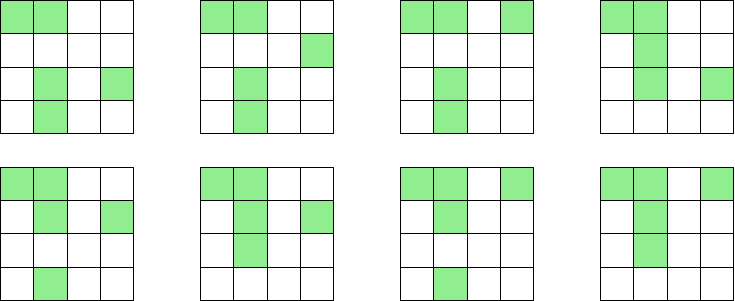}
		\caption{The eight diagrams in $\Koh(w)$ for $w=31542$.}
		\label{fig:kohnert-31542}
	\end{figure}
	
	\begin{theorem}[{\cite{original_kohnert,winkel1,winkel2}}]
		\label{thm:kohnert_rule}
		For any permutation $w\in S_n$ and $\alpha\in \Z_{\geq 0}^n$, 
		\[[x^\alpha]\mathfrak{S}_w = \#\Koh(w,\alpha).\]
	\end{theorem}
	
	Kohnert's rule has seen renewed interest, with the discovery of new proofs (\cite{assaf_kohnert,assaf_flagged_schur}), the study of poset-theoretic questions (\cite{kohnert_poset_shellability,ranked_bounded_kohnert_posets,ghost_kohnert_posets,koh_posets_ne_diagrams}), and some questions regarding minimal Kohnert move sequences (\cite{comb_puzzle_lascoux,comb_puzzle_kohnert}).
	
	\FloatBarrier
	
	\section{The Ross--Yong and Robichaux Conjectures}
	\label{sec:conjectures}
	
	We describe two proposed \K-theoretic generalizations of Kohnert's rule, one due to Ross--Yong \cite{ghost_kohnert}, and one due to Robichaux \cite{ghost_kohnert_fix}. We discuss counterexamples to both rules. We make a weaker conjecture regarding Robichaux's rule (Conjecture~\ref{conj:support}). 
	
	All computations supporting the enumerative claims in this section were carried	out using Dave Anderson's Julia package \cite{julia}.
	
	\begin{definition}
		A \newword{grave diagram} is a pair $E=(B,G)$ of disjoint subsets of
		$[n]\times[n]$. Cells in $B$ are called \newword{boxes}, and cells in $G$ are called
		\newword{ghosts}. We set $\#E=\#B+\#G$ and 
		\[
			\wt(E)_i=\#\{j:(i,j)\in B\cup G\}.
		\]
		We call arbitrary elements of $B\cup G$ \newword{occupied cells} of $E$.
	\end{definition}
	
	\begin{definition}[{\cite{ghost_kohnert}}]
		\label{def:k_koh_moves}
		Let $E=(B,G)$ be a grave diagram, and suppose $(i,j)\in B$ is rightmost among all
		occupied cells in row $i$. Set
		\[
			i' = \max\{r \mid 1\leq r < i \mbox{ and } (r,j)\notin B\cup G\}.
		\]
		Assume $i'$ exists and every occupied cell $(r,j)$ with $i'<r\leq i$ is a box.
		\begin{itemize}
			\item The \newword{(ordinary) Kohnert move} on $E$ at row $i$ is
			\[(B,G)\longmapsto \left((B\setminus\{(i,j)\})\cup\{(i',j)\},\,G\right).\]
			\item The \newword{\K-Kohnert move} on $E$ at row $i$ is
			\[(B,G)\longmapsto\left((B\setminus\{(i,j)\})\cup\{(i',j)\},\,G\cup\{(i,j)\}\right).\]
		\end{itemize}
		Let $\KKoh(w)$ be the closure of $(D(w),\varnothing)$ under these moves. Set
		\[\KKoh(w,\alpha)=\left\{E\in \KKoh(w)\mid \wt(E)=\alpha\right\}.\]
		We will refer to the elements of $\KKoh(w)$ as \newword{\K-Kohnert diagrams}.
	\end{definition}
	
	 In a Kohnert move on $E$ in row $i$, we view the rightmost box in row $i$ as jumping upward in its column, additionally leaving behind a ghost in a \K-Kohnert move. Note that the definitions of the moves allow boxes to jump over other boxes but not over ghosts. Additionally, ghosts cannot move.
	
	\begin{example}
		Figure~\ref{fig:k_kohnert_example} shows an example of an ordinary Kohnert move and a \K-Kohnert move on a grave diagram. Note that we picture boxes in green and ghosts in gray.
	\end{example}
	
	\begin{figure}[ht]
		\centering
		\includegraphics[scale=1]{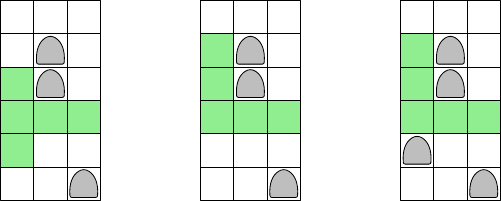}
		\caption{A grave diagram (left), and the results of an ordinary Kohnert move (middle) and a \K-Kohnert move (right).}
		\label{fig:k_kohnert_example}
	\end{figure}
	
	\begin{conjecture}[{\cite[Conjecture 1.6]{ghost_kohnert}}]
		\label{conj:ry_kohnert}
		For any permutation $w\in S_n$,
		\[\mathfrak{G}_w=\sum_{E\in \KKoh(w)} (-1)^{|\wt(E)|-\ell(w)}x^{\wt(E)}.\]
		In other words,
		\[\#\KKoh(w,\alpha) = \#\PD(w,\alpha)\quad\mbox{ for any }\alpha\in \Z_{\geq 0}^n.\]
	\end{conjecture}
	
	Ross and Yong made an analogous conjecture for Lascoux polynomials, an inhomogeneous version of Demazure characters (\cite[Conjecture 1.4]{ghost_kohnert}). Their Lascoux conjecture was proven by Pan and Yu \cite{lascoux_k_kohnert} via a bijection to certain set-valued tableaux defined by Shimozono and Yu \cite{groth_to_lascoux}.
	
	It was shown by Robichaux in \cite[Claim 3.4]{ghost_kohnert_fix} that Conjecture~\ref{conj:ry_kohnert} fails for some permutations in $S_8$. Robichaux gave the minimal counterexample $w=12365847$ and $\alpha=(3,3,3,2)$ as an explicit counterexample, with 
	\[\#\PD(w,\alpha)=3\quad\mbox{and}\quad\#\KKoh(w,\alpha)=2.\]
	Robichaux also provided $w=12375846$ as an example showing the Ross--Yong rule fails even for vexillary permutations. Note that both $12\underline{365}8\underline{4}7$ and $12\underline{375}8\underline{4}6$ contain a $1432$-pattern (underlined).
	
	Robichaux proposed the following enlargement of $\KKoh(w)$ to correct the undercount of the Ross--Yong rule:
	
	\begin{definition}
		\label{def:ghost_k_koh_moves}
		Let $E=(B,G)$ be a grave diagram. Suppose $(i,j)\in G$ is the rightmost occupied cell in row $i$, and set
		\[i' = \max\{r \mid 1\leq r < i \mbox{ and } (r,j)\notin B\cup G\}.\]
		Assume $i'$ exists and every occupied cell $(r,j)$ with $i'<r<i$ is a ghost.

		\begin{itemize}
			\item The \newword{ghost Kohnert move} on $E$ at row $i$ is
			\[(B,G)\longmapsto \left(B,\,(G\setminus\{(i,j)\})\cup\{(i',j)\}\right).\]
			
			\item The \newword{ghost \K-Kohnert move} on $E$ at row $i$ is
			\[(B,G)\longmapsto \left(B,\,G\cup\{(i',j)\}\right).\]
		\end{itemize}
		Let $\GKKoh(w)$ denote the closure of $(D(w),\varnothing)$ under ordinary, \K-, ghost, and ghost \K-Kohnert moves. Set $\GKKoh(w,\alpha)=\left\{E\in \GKKoh(w)\mid \wt(E)=\alpha\right\}$. We identify grave diagrams having $G=\varnothing$ with usual diagrams.
	\end{definition}
	
	We will refer to the elements of $\GKKoh(w)$ as ghost \K-Kohnert diagrams.
	
	\begin{example}
		Figure~\ref{fig:ghost_k_kohnert_example} shows an example of a ghost Kohnert move and a ghost \K-Kohnert move on a grave diagram. Note that we picture boxes in green and ghosts in gray.
	\end{example}
	
	\begin{figure}[ht]
		\centering
		\includegraphics[scale=1]{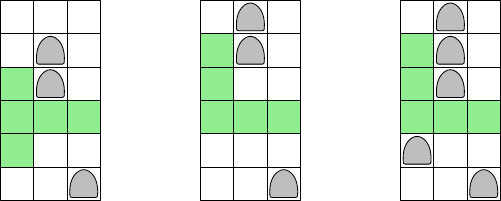}
		\caption{A grave diagram (left), and the results of a ghost Kohnert move (middle) and a ghost \K-Kohnert move (right).}
		\label{fig:ghost_k_kohnert_example}
	\end{figure}
	
	\begin{conjecture}[{\cite[Conjecture 4.3]{ghost_kohnert_fix}}]
		\label{conj:cr_kohnert}
		For any permutation $w\in S_n$ and $\alpha\in \Z_{\geq 0}^n$,
		\[\#\GKKoh(w,\alpha) = \#\PD(w,\alpha).\]
	\end{conjecture}
	
	We present counterexamples to Conjecture~\ref{conj:cr_kohnert} in both numerical directions.
	
	\begin{proposition}
		\label{prop:ghost_koh_overcounts}
		The permutation $v=142396857$ and exponent $\gamma=(4,6,3,3,0,0,0,0,0)$ witness
		\[
		\#\PD(v,\gamma)=1
		\quad\text{and}\quad
		\#\GKKoh(v,\gamma)\geq 2.
		\]
	\end{proposition}
	
	\begin{proof}
		The Rothe diagram of $v$ and a pipe dream of $v$ with weight $\gamma$ is shown in Figure~\ref{fig:142396857_4633_pd}. Direct computational enumeration gives $\#\PD(v,\gamma)=1$.
		
		\begin{figure}[ht]
			\begin{center}
				\includegraphics[scale=.75]{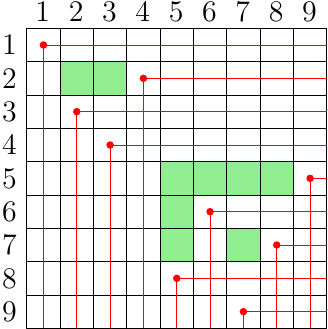}
				\qquad\qquad\qquad
				\includegraphics[scale=.75]{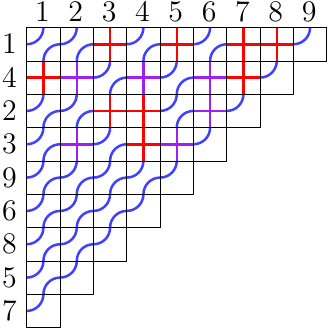}
				\caption{The Rothe diagram of $v=142396857$ (left) and unique pipe dream of $v$ with weight $\gamma=(4,6,3,3)$ (right).}
				\label{fig:142396857_4633_pd}
			\end{center}
		\end{figure}
		
		On the other hand, the two diagrams in Figure~\ref{fig:142396857_4633_ghost_koh_overcounts} both lie in $\GKKoh(v,\gamma)$. We demonstrate how to achieve these diagrams in Figures~\ref{fig:142396857_4633_ghost_koh_1_steps} and~\ref{fig:142396857_4633_ghost_koh_2_steps} (dropping empty columns and bottommost empty rows for compactness).
	\end{proof}
	
	\begin{figure}[ht]
		\begin{center}
			\includegraphics[scale=.75]{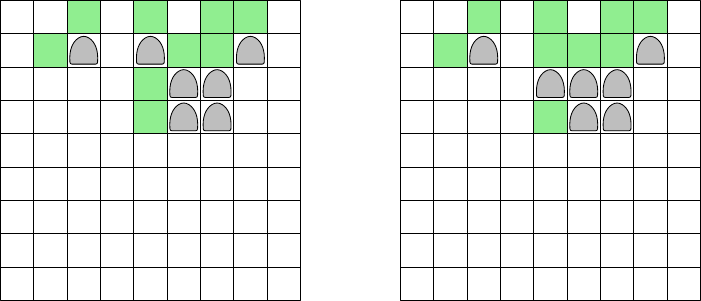}
			\caption{Witnesses to $\#\GKKoh(v,\gamma)\geq 2$.}
			\label{fig:142396857_4633_ghost_koh_overcounts}
		\end{center}
	\end{figure}
	
	\begin{proposition}
		\label{prop:undercount_proof}
		The permutation $u=123765948$ and exponent $\beta=(4,4,4,4,1,0,0,0,0)$ witness
		\[
		\#\GKKoh(u,\beta)=2
		\quad\text{and}\quad
		\#\PD(u,\beta)=3.
		\]
	\end{proposition}
	
	\begin{figure}[ht]
		\begin{center}
			\includegraphics[scale=.75]{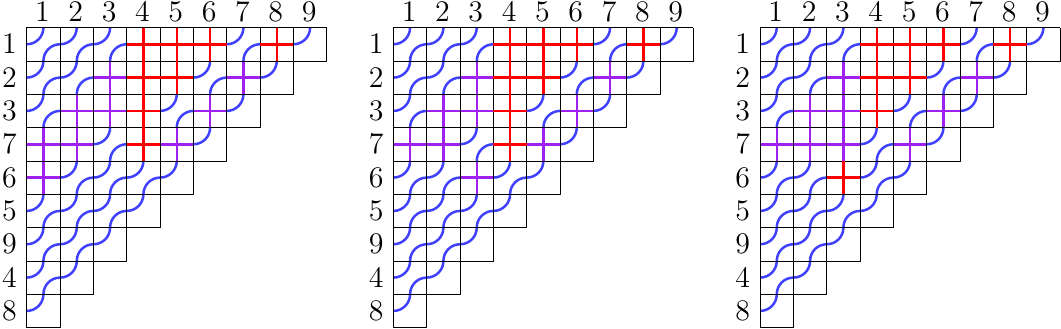}
			\caption{The three pipe dreams of $u=123765948$ with weight $\beta=(4,4,4,4,1)$.}
			\label{fig:123765948_44441_pd}
		\end{center}
	\end{figure}
	
	\begin{figure}[ht]
		\begin{center}
			\includegraphics[scale=.75]{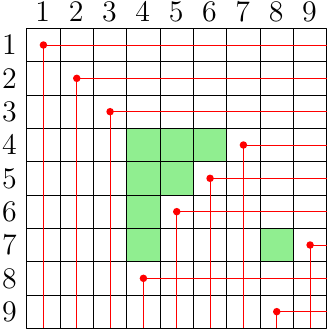}
			\qquad\qquad
			\includegraphics[scale=.75]{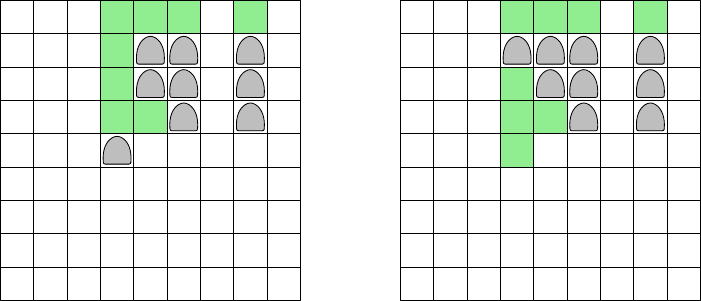}
			\caption{The Rothe diagram of $u=123765948$ (left) and witnesses to $\#\GKKoh(u,\beta)\geq 2$ (right) for $\beta=(4,4,4,4,1)$.}
			\label{fig:123765948_44441_ghost_koh_undercounts}
		\end{center}
	\end{figure}
	
	\begin{proof}
		A direct computation of $\mathfrak{G}_u$ gives
		\[\#\PD(u,\beta) = \left|[x^\beta]\mathfrak{G}_u\right|=3. \]
		The three pipe dreams of $u$ with weight $\beta$ are shown in Figure~\ref{fig:123765948_44441_pd}. The Rothe diagram of $u$ and two ghost \K-Kohnert diagrams of $u$ with weight $\beta$ are shown in Figure~\ref{fig:123765948_44441_ghost_koh_undercounts}. Sequences of moves generating the diagrams in Figure~\ref{fig:123765948_44441_ghost_koh_undercounts} are given in Figures~\ref{fig:123765948_44441_ghost_koh_1_steps} and~\ref{fig:123765948_44441_ghost_koh_2_steps}. For completeness, we include a direct proof that $\#\GKKoh(u,\beta)=2$. Let $E\in\GKKoh(u,\beta)$. 
		
		\medskip
		
		\noindent \textbf{Claim 1:} The support of $E$ is determined. 
		
		\medskip
		
		\noindent Since the four types of Kohnert move all preserve column index, all occupied cells of $E$ lie in columns $4,5,6,8$. By the weight condition, rows $1,2,3,4$ contain one occupied cell in each of these columns. There is one additional occupied cell $C$ somewhere in row $5$ of $E$. We show $C$ can only be in column $4$.
		
		Observe that $C$ cannot be in column $6$, since column $6$ has no initial box at or below row $5$. Suppose $C$ is in column $5$, so at position $(5,5)$. The initial box at position $(5,5)$ in $D(u)$ must be removed in order for the box at position $(5,4)$ to move out of row $5$. However $(5,5)\in D(u)$ is lowest in its column, so it must move via \K-Kohnert move in order for $C$ to ever be occupied again. But then the cell $(5,5)$ would have to remain occupied by a ghost, and the initial box at position $(5,4)$ can never move.
		
		Finally, $C$ cannot be in column $8$. If it were, then column 4 of $E$ must consist entirely of (non-ghost) boxes. Up to and including the moment when it reaches row $4$, each time the initial box at $(7,8)$ moves there will be a box in column 4 in a strictly lower row. When the initial box from $(7,8)$ passes into row 4, it must leave a ghost in row $5$. This will prevent at least one box in column $4$ from reaching rows $\{1,2,3,4\}$. Hence the target weight cannot be achieved if $C$ is in column $8$. 
		
		Consequently, the only viable possibility is $C$ in column $4$, implying that the support of $E$ is exactly $\left(\{1,2,3,4\}\times\{4,5,6,8\}\right)\cup\{(5,4)\}$. 
		
		\medskip
		
		\noindent \textbf{Claim 2:} $E$ equals one of the diagrams in Figure~\ref{fig:123765948_44441_ghost_koh_undercounts}.
		
		\medskip		
		
		The number of ordinary boxes in each column is invariant, so columns $4,5,6,8$ contain respectively $4,2,1,1$ ordinary boxes in $E$. In a column with a single ordinary box, all ghosts created from it lie in rows below it; hence columns $6$ and $8$ of $E$ coincide with those of the diagrams in Figure~\ref{fig:123765948_44441_ghost_koh_undercounts}. 
		
		Next consider column $5$. Note that the cell $(4,6)$ must remain occupied throughout the entire move sequence. It follows that the initial boxes $(4,4)$ and $(4,5)$ can never move. In particular, cell $(4,5)$ must remain occupied by an ordinary box. Since column $5$ must be completely occupied in rows $1,2,3,4$, its only possible configuration is the one shared by the two diagrams in Figure~\ref{fig:123765948_44441_ghost_koh_undercounts}. 
		
		It remains to show column 4 of $E$ must have either of the configurations in Figure~\ref{fig:123765948_44441_ghost_koh_undercounts}. The reader is advised to consult Figure~\ref{fig:123765948_undercount} throughout this argument.
		
		\begin{figure}[ht]
			\begin{center}
				\includegraphics[scale=1.3]{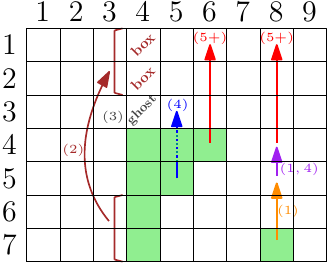}
				\caption{The forced construction steps in the proof of Claim~2 of Proposition~\ref{prop:undercount_proof}. The steps $(1)$, $(2)$, $(3)$, $(4)$ occur in labeled order, with the step $(1,4)$ occurring any time after step $(1)$ and before step $(4)$. The steps $(5+)$ must occur afterwards, but are at odds with the required end result.}
				\label{fig:123765948_undercount}
			\end{center}
		\end{figure}
		
		Since column $4$ must consist of $4$ boxes and $1$ ghost, it is enough to show the ghost lies in row $2$ or row $5$. It is immediate from Definitions~\ref{def:k_koh_moves} and~\ref{def:ghost_k_koh_moves} that the top occupied cell in any column of a ghost Kohnert diagram is an ordinary box. To see that the ghost cannot lie in row $4$, recall we noted in the previous paragraph that cell $(4,4)$ remains occupied by an ordinary box throughout the entire move sequence.
		
		It remains to rule out the possibility that the unique ghost in column $4$ lies in row $3$. Suppose, for contradiction, that this is the case. Since $(4,4)$ is a permanent ordinary box, the ghost at $(3,4)$ cannot be produced by a ghost move from a lower row. So the ghost at $(3,4)$ must be created by a \K-Kohnert move shifting a box occupying cell $(3,4)$.
		
		At some point then, the cell $(3,4)$ must be occupied and rightmost in its row.
		
		Consider the initial box $(5,5)\in D(u)$. Since $(4,5)$ is fixed, the first move of that box shifts it to $(3,5)$. After that first move, we claim the cell $(3,5)$ remains occupied forever. To see this, note that if $(3,5)$ were ever empty afterwards, then it would remain so. This would contradict the established support of $E$. It follows that the ghost at $(3,4)$ must be created before the first move of the initial box $(5,5)$, for otherwise $(3,5)$ would already be occupied, and $(3,4)$ could never be rightmost in row $3$.
		
		Once the ghost at $(3,4)$ has been created, no ordinary box in column $4$ lying below row $3$ can ever move above row $3$. Hence,		before the first move of $(5,5)$, the two column $4$ boxes that eventually occupy rows $1$ and $2$ must already have crossed above row $3$. Since
		$(4,4)$ is fixed, and since $(5,4)$ is still blocked by the occupied cell
		$(5,5)$ at that time, these two boxes must be the initial boxes in rows $6$
		and $7$ of column $4$.
		
		For these two column $4$ boxes to cross above row $3$, the initial box at $(7,8)$ must first move to $(5,8)$. Note also that the first move of the initial box $(5,5)$ requires $(5,8)$ to be empty, so the initial box at $(7,8)$ must move to $(4,8)$ prior to this.
		
		After these moves have been completed, it is now impossible for both of columns $6$ and $8$ to reach their final configuration with rows 1,2,3,4 all filled. See Figure~\ref{fig:123765948_undercount} for a visual representation of these events. This precludes the ghost in column $4$ from appearing in row $3$, leaving exactly the two configurations shown in Figure~\ref{fig:123765948_44441_ghost_koh_undercounts}.
	\end{proof}
	
	\begin{remark}
		We note (without proof) that the vexillary case of Robichaux's rule appears no simpler than the general case. One can check that the vexillary permutations $123486957$ and $123486975$ witness overcounts of Robichaux's \K-Kohnert rule. Similarly, the vexillary permutations $123875946$ and $123875964$ witness undercounts.
	\end{remark}
	
	\subsection{Support of Grothendieck Polynomials}\phantom{}\newline\vspace{-2ex}
	
	We note an even more extreme failing of the Ross--Yong rule.
	
	\begin{proposition}
		\label{prop:RY-support-failure}
		For $w=123764958$ and $\alpha=(4,4,4,3,0,0,0,0,0)$, one has
		\[
		\#\PD(w,\alpha)=1,
		\quad\mbox{but}\quad
		\#\KKoh(w,\alpha)=0.
		\]
	\end{proposition}
	
	\begin{proof}
		The equalities $\#\PD(w,\alpha)=1$ and $\#\KKoh(w,\alpha)=0$ can be verified by exhaustive enumeration. The witnessing pipe dream and Rothe diagram of $w$ are given in Figure~\ref{fig:123764958_4443_ry_failure}.
	\end{proof}
	
	\begin{figure}[ht]
		\begin{center}
			\includegraphics[scale=.75]{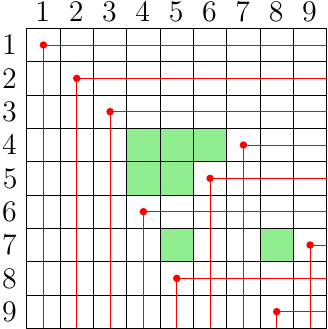}
			\qquad\qquad
			\includegraphics[scale=.75]{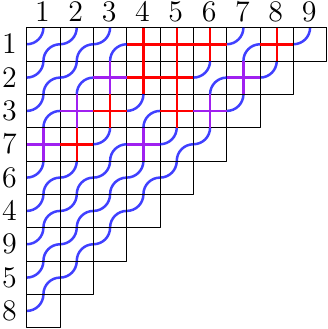}
			\caption{The Rothe diagram of $w=123764958$ (left), and a witness to $\#\PD(w,\alpha)=1$ (right) for $\alpha=(4,4,4,3)$.}
			\label{fig:123764958_4443_ry_failure}
		\end{center}
	\end{figure}
	
	Robichaux's ghost \K-Kohnert model does produce a diagram of weight $\alpha$,	shown in Figure~\ref{fig:ry_support_correction}. A sequence of explicit moves yielding this diagram is given in Figure~\ref{fig:123764958_4443_support_ghost_koh}.
	
	\begin{figure}[ht]
		\centering
		\includegraphics[scale=.75]{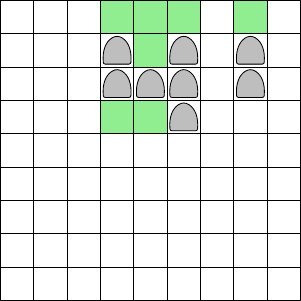}
		\caption{A diagram in $\GKKoh(w)\setminus\KKoh(w)$ for $w=123764958$.}
		\label{fig:ry_support_correction}
	\end{figure}
	
	Proposition~\ref{prop:RY-support-failure} shows that the original Ross--Yong rule can miss support monomials. Exploratory computations found no instance of a similar failing for Robichaux's rule. This motivates the following weakening of Conjecture~\ref{conj:cr_kohnert}.
	
	\begin{conjecture}
		\label{conj:support}
		For any permutation $w\in S_n$,
		\[\left\{\wt(P) \mid P\in\PD(w)\right\} = \left\{\wt(E)\mid E\in \GKKoh(w)\right\}.\]
	\end{conjecture}
	
	In words, Robichaux's ghost \K-Kohnert rule correctly computes the support of any Grothendieck polynomial. We have checked Conjecture~\ref{conj:support} for all $w\in S_9$.
	
	
	\section{1432-Avoiding Permutations}
	\label{sec:1432}
	
	In \cite[Theorem 5.1]{ghost_kohnert_fix}, Robichaux proved that both the ghost rule (Conjecture~\ref{conj:cr_kohnert}) and the Ross--Yong rule (Conjecture~\ref{conj:ry_kohnert}) hold for 321-avoiding permutations. This was done through reduction to a special case of the Pan--Yu bijection \cite{lascoux_k_kohnert}. Robichaux also suggested both conjectures hold for 1432-avoiding permutations, and may admit tableaux-theoretic proofs for that class via a formula of Fan--Guo \cite{fg_svrt}.
	
	Following Robichaux's suggestion, we construct a (nonrecursive) weight-preserving bijection between Fan--Guo's set-valued Rothe tableaux and Ross--Yong's \K-Kohnert diagrams. Note that our construction is independent of the Pan--Yu bijection.
	
	\subsection{Fan--Guo's Tableaux Formula}\phantom{}\newline\vspace{-2ex}
	
	\begin{definition}
		\label{def:svrt}
		A \newword{flagged set-valued Rothe tableau} of shape $D(w)$ is a map
		\[
		T:D(w)\to 2^{[n]}\setminus\{\varnothing\}
		\]
		such that (whenever the relevant cells lie in $D(w)$)
		\begin{align}
			\min \, T(r,c)&\geq \max \, T(r,c') \quad \text{if } c<c', \label{eq:row-cond}\\
			\max \, T(r,c)&<\min \, T(r',c) \quad\text{if } r<r', \label{eq:col-cond}\\
			\max \, T(r,c)&\leq r. \label{eq:flag-cond}
		\end{align}
		We denote by $\SVRT(w)$ the collection of set-valued Rothe tableaux of shape $D(w)$. The weight of $T\in \SVRT(w)$ is the vector with components
		\[\wt(T)_i=\#\{(r,c)\in D(w)\mid i\in T(r,c)\}.\]
	\end{definition}
	
	\begin{theorem}[{\cite[Corollary 1.2]{fg_svrt}}]
		\label{thm:fg_formula}
		If $w\in S_n$ is $1432$-avoiding, then
		\[
			\mathfrak{G}_w(x)=\sum_{T\in\SVRT(w)}(-1)^{|\wt(T)|-\ell(w)}x^{\wt(T)}.
		\]
	\end{theorem}
	
	\begin{example}
		For $w=1423$, the set $\SVRT(w)$ is given in Figure~\ref{fig:1423-svrt}.
	\end{example}
	
	\begin{figure}[ht]
		\centering
		\includegraphics{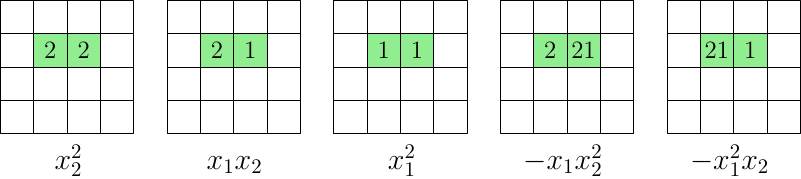}
		\caption{The set $\SVRT(w)$ for $w=1423$. The corresponding signed monomials in $\mathfrak{G}_w$ are shown below each tableau.}
		\label{fig:1423-svrt}
	\end{figure}
	
	For $w$ 1432-avoiding, we construct a weight-preserving bijection
	\[
	\SVRT(w)\longleftrightarrow \KKoh(w).
	\]
	
	\subsection{Tableaux to Grave Diagrams}\phantom{}\newline\vspace{-2ex}
	
	\begin{definition}
		\label{def:theta}
		For $T\in \SVRT(w)$ and $(r,c)\in D(w)$, write
		\[
		T(r,c)=\{a_1<a_2<\cdots<a_m\}.
		\]
		Define $\Theta_w(T)=(B_T,G_T)$ by placing a box at $(a_1,c)$ and ghosts at
		$(a_2,c),\ldots,(a_m,c)$. Explicitly:
		\begin{align*}
			B_T&=\left\{\left(\min \, T(r,c),\,c\right)\mid(r,c)\in D(w)\right\},\\
			G_T&=\left\{(a,c)\mid \mbox{ there exists } (r,c)\in D(w) \mbox{ with } a\in T(r,c) \mbox{ and }a>\min \, T(r,c) \right\}.
		\end{align*}
	\end{definition}
	
	\begin{example}
		\label{exp:theta_example}
		An SVRT $T$ for $w=12365847$ and the corresponding grave diagram $\Theta_w(T)$ are shown in Figure~\ref{fig:svrt_to_grave_diagram}.
	\end{example}
	
	\begin{figure}[ht]
		\centering
		\includegraphics[scale=1]{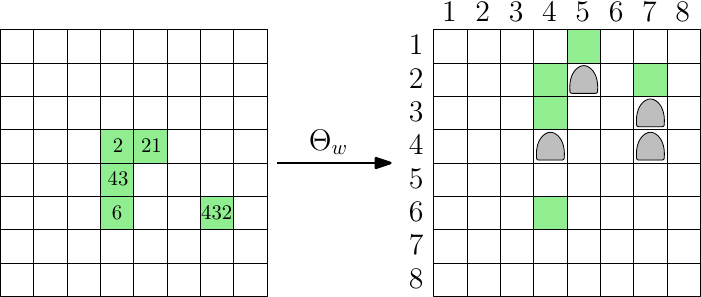}
		\caption{An example of the map $\Theta_w$.}
		\label{fig:svrt_to_grave_diagram}
	\end{figure}
	
	\begin{lemma}
		\label{lem:theta-welldefined}
		If $T\in\SVRT(w)$, then $\Theta_w(T)$ is a grave diagram.
	\end{lemma}
	
	\begin{proof}
		Cells of $D(w)$ in different columns give boxes/ghosts in different columns of $\Theta_w(T)$. So issues with $\Theta_w(T)$ could only arise from cells in the same column of $D(w)$. Fix a column $c$ and $(r,c),(r',c)\in D(w)$ with $r<r'$. The column condition (\ref{eq:col-cond}) gives
		\[
			\max \, T(r,c)<\min \, T(r',c),
		\]
		so the sets $T(r,c)$ and $T(r',c)$ are disjoint. Thus two different cells of $D(w)$ in the same column cannot produce the same occupied cell in $\Theta_w(T)$. Clearly $B_T\cap G_T=\varnothing$.
	\end{proof}
	
	Observe that the map $\Theta_w$ is weight-preserving:
	\begin{equation}
		\label{eq:theta-weight}
		\wt(\Theta_w(T))=\wt(T).
	\end{equation}
	
	\medskip
	
	\subsection{Grave Diagrams to Tableaux}\phantom{}\newline\vspace{-2ex}
	
	In the following definitions, let $E=(B,G)$ be a grave diagram.
	
	\begin{definition}
		Fix a column $c$ of $E$, and suppose $c$ has $k$ occupied cells. Form the \newword{column word} $W_c(E)\in \{\blacksquare,\boxtimes\}^k$ by reading the occupied cells of column $c$ from north to south ignoring empty spaces. Record $\blacksquare$ for each box encountered, and $\boxtimes$ for each ghost.
	\end{definition}
	
	\begin{definition}
		We say column $c$ of $E$ is \newword{threaded} if 
		\[W_c(E) = (\blacksquare\,\boxtimes^{k_1})\cdots (\blacksquare\,\boxtimes^{k_p})\]
		for some $k_1,\ldots,k_p\geq 0$. In this case, call the blocks $(\blacksquare\,\boxtimes^{k_j})$ the \newword{threads} of column $c$. We treat empty columns as being vacuously threaded, but having no threads.
	\end{definition}
	
	In other words, a thread consists of a box followed by all ghosts prior to the next box downwards in the column (if present, else the end of the column). A (nonempty) column can only fail to be threaded if its topmost occupied cell is a ghost.
	
	\begin{definition}
		Call $E$ \newword{thread-compatible with $w$} if every column of $E$ is threaded and for each $c$, the number of threads in column $c$ of $E$ equals the number of cells in column $c$ of $D(w)$.
	\end{definition}
	
	\begin{definition}
		\label{def:xi}
		Let $E$ be thread-compatible with $w\in S_n$. From north to south, assign the threads in column $c$ of $E$ to the cells of $D(w)$ in column $c$. Define a filling 
		\[\Xi_w(E):D(w)\to 2^{[n]}\setminus\{\varnothing\}\]
		as follows: when the thread assigned to $(r,c)\in D(w)$ occupies rows indexed by $S$ in $E$, set
		$\Xi_w(E)(r,c)=S$.
	\end{definition}
	
	\begin{example}
		Reversing Example~\ref{exp:theta_example}, we show in Figure~\ref{fig:grave_diagram_to_svrt} a grave diagram $E$ thread-compatible with $w=12365847$ and the corresponding filling $\Xi_w(E)$ of $D(w)$. The threads in each column of $E$ are indicated with outlines.
	\end{example}
	
	\begin{figure}[ht]
		\centering
		\includegraphics[scale=1]{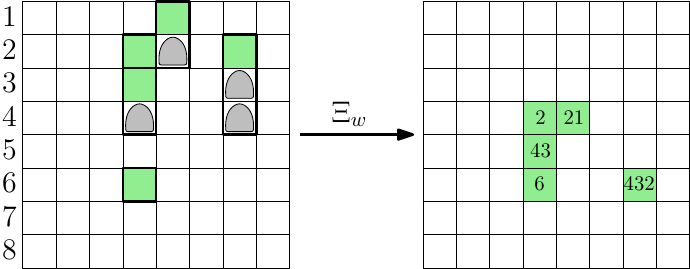}
		\caption{An example of the map $\Xi_w$.}
		\label{fig:grave_diagram_to_svrt}
	\end{figure}

	
	\subsection{The Easy Inclusion} \phantom{}\newline\vspace{-2ex}
	
	To work towards the bijection
	\[
		\begin{tikzcd}
			\SVRT(w) \arrow[r, bend left, "\Theta_w"] & \KKoh(w) \arrow[l, bend left, "\Xi_w"]
		\end{tikzcd}
	\]
	we first prove the inclusion $\Theta_w(\SVRT(w)) \subseteq \KKoh(w)$. We begin with the following technical lemma describing the local properties of $\Theta_w$ within each column.
	
	\begin{lemma}
		\label{lem:same-column-isolation}
		Let $T\in\SVRT(w)$. Choose $r$ maximal such that some cell in row $r$ is
		not filled by $\{r\}$, and choose $c$ minimal such that $T(r,c)\neq \{r\}$. Set
		$p=\min \, T(r,c)$. Consider the grave diagram $E=\Theta_w(T)$, with boxes $B$ and ghosts $G$.
		
		Then every occupied cell $(a,c)\in B\cup G$ with $p+1\leq a\leq r$ satisfies 
		\begin{itemize}
			\item $a\in T(r,c)$;
			\item $a\notin T(r',c)$ for any $(r',c)\in D(w)$ with $r'\neq r$.
		\end{itemize}
		
		Consequently $p<r$, the smallest $q>p$ for which $(q,c)\notin B$ is $q=p+1\leq r$, and if $(p+1,c)\in G$, then it lies in the same thread as $(p,c)$.
	\end{lemma}
	
	\begin{proof}
		If $p=r$, then the flag condition (\ref{eq:flag-cond}) forces every entry of $T(r,c)$ to be equal to $r$, contrary to $T(r,c)\neq\{r\}$. Suppose $(r',c)\in D(w)$ with $r\neq r'$ and $a\in T(r',c)$. If $r'<r$, then column-strictness gives 
		\[a\leq \max \, T(r',c)<\min \, T(r,c)=p.\]
		If $r'>r$, then the maximal choice of $r$ forces $T(r',c)=\{r'\}$, so $a=r'>r$. The two bullet points of the lemma follow.
		
		The set $T(r,c)$ contributes a box only in row $p$ of $E$; its other entries (if any),
		contribute ghosts in the same thread as $(p,c)$. By the two bullet points, $(p+1,c)\notin B$, and $(p+1,c)\in G$ would clearly belong to the same thread as $(p,c)$.
	\end{proof}

	\begin{proposition}
		\label{prop:theta-in-kkoh}
		For every $w\in S_n$ and every $T\in\SVRT(w)$, one has
		\[\Theta_w(T)\in \KKoh(w).\]
	\end{proposition}
	
	\begin{proof}
		For fixed $w$, we prove the statement by induction on the nonnegative integer
		\[
		\mathcal{I}(T)\coloneqq
		\sum_{(i,j)\in D(w)}\left(\sum_{a\in T(i,j)}(i-a)\right)
		\,\,
		+\sum_{(i,j)\in D(w)}\left(\#T(i,j)\,-1\right).
		\]
		The flag condition (\ref{eq:flag-cond}) implies $a\leq i$ for every $a\in T(i,j)$, so
		$\mathcal{I}(T)\geq 0$. If $\mathcal{I}(T)=0$, then every cell has exactly one entry, the row index of the cell. 
		Hence $\mathcal{I}(T)=0$ implies $T(i,j)=\{i\}$ for all $(i,j)\in D(w)$, and
		\[\Theta_w(T)=(D(w),\varnothing)\in\KKoh(w).\]
		
		Assume now that $\mathcal{I}(T)>0$. Choose $r$ maximal such that some cell in row
		$r$ is not filled by $\{r\}$, and then choose $c$ minimal such that
		$T(r,c)\neq \{r\}$. Put
		\[p=\min \, T(r,c),\qquad q=p+1,\qquad E=\Theta_w(T).\]
		By Lemma~\ref{lem:same-column-isolation}, we have $q\leq r$, and the cell
		$(q,c)$ of $E$ is either empty or is a ghost in the same thread as the box $(p,c)$ (depending on whether $q\in T(r,c)$).
		
		We define a new tableau $T'$ as follows. All cells except $(r,c)$ retain their original filling. For $(r,c)$, define
		\[T'(r,c) = \left(T(r,c)\setminus\{p\}\right)\cup\{q\}\]		

		Regardless of whether or not $q\in T(r,c)$, it follows that 
		\[
		\min \, T'(r,c)=q
		\quad\text{and}\quad
		\max \, T'(r,c)\leq r.
		\]
		We first check that $T'\in\SVRT(w)$. The flag condition for $T'$ only needs to be checked at $(r,c)$, where it clearly still holds. For the column condition (\ref{eq:col-cond}), we need only check column $c$. If $(s,c)\in D(w)$ with $s<r$, then
		\[
		\max \, T'(s,c)=\max \, T(s,c)<\min \, T(r,c)=p<q=\min \, T'(r,c).
		\]
		Consider $(s,c)\in D(w)$ with $s>r$. By the maximal choice of $r$, it follows that $T(s,c)=\{s\}$. But the flag condition forces
		\[\max \, T'(r,c)\leq r<s=\min \, T'(s,c).\]
		Hence the column condition still holds for $T'$.
		
		To verify the row condition on $T'$, we need only check row $r$. Let $(r,d)\in D(w)$ with $d>c$. The row
		condition on $T$ gives
		\[
		\max \, T'(r,d) = \max \, T(r,d)\leq \min \, T(r,c)=p<q=\min \, T'(r,c).
		\]
		Next consider $(r,d)\in D(w)$ with $d<c$. By the minimal choice of $c$, we have $T(r,d)=\{r\}$. But then the flag condition gives
		\[
		\min \, T'(r,d)=r\geq \max \, T'(r,c).
		\]
		Thus $T'\in\SVRT(w)$. 
		
		By construction, $\mathcal{I}(T')<\mathcal{I}(T)$. Let $E'=\Theta_w(T')$. By induction, $E'\in\KKoh(w)$. We will show that $E$ is obtained from $E'$ by one legal Kohnert or \K-Kohnert move shifting a box in column $c$.
		
		In $E'$, the cell $(q,c)=(p+1,c)$ is a box and the cell $(p,c)$ is empty. It remains only to verify that $(q,c)$ is the rightmost occupied cell in row $q$ of $E'$. Suppose not. Then there is an occupied cell $(q,d)$ of $E'$ with $d>c$. Since $T$ and $T'$ differ only at $(r,c)$, there exists a Rothe cell $(s,d)\in D(w)$ such that $q\in T(s,d)=T'(s,d)$.
		
		If $s=r$, then the row condition in row $r$ of $T$ gives
		\[q\leq \max \, T(r,d)\leq \min \, T(r,c)=p,\]
		contradicting $q=p+1$. If $s<r$, then $(r,c),(s,d)\in D(w)$ with $s<r$ and $c<d$. By the northwest property of Rothe diagrams, $(s,c)\in D(w)$.
		
		But consider the column condition in column $c$ of $T$ gives
		\[\max \, T(s,c)<\min \, T(r,c)=p,\]
		while the row condition in row $s$ of $T$ gives
		\[\min \, T(s,c)\geq \max \, T(s,d)\geq q.\]
		Hence
		\[
		q\leq \min \, T(s,c)\leq \max \, T(s,c)<p<q,
		\]
		a contradiction. Lastly, suppose $s>r$. By the maximal choice of $r$, necessarily $T(s,d)=\{s\}$. But then
		$q\in T(s,d)$ implies $q=s$, contradicting $q\leq r<s$. 
		
		These contradictions show there can be no such cell $(s,d)\in D(w)$, and therefore no such occupied cell $(q,d)\in E'$. Then the box $(q,c)$ is the rightmost occupied cell in row $q$ of $E'$, and $E$ is obtained from $E'$ by a (Kohnert if $q\notin T(r,c)$, else \K-Kohnert) move taking the box $(q,c)$ up one row to $(p,c)$. Thus $E\in\KKoh(w)$, completing the induction.
	\end{proof}
	
	\begin{remark}
		\label{rmk:simple_1}
		Note that in the previous proof, the Kohnert move $E'\longrightarrow E$ has a box moving up only a single row (row $q=p+1$ to row $p$). We will utilize this fact about the proof when we study such \newword{simple} moves in Section~\ref{sec:simple}. See the proof of Lemma~\ref{lem:simple_theta_codomain}.
	\end{remark}

	\subsection{The Hard Inclusion}\phantom{}\newline\vspace{-2ex}
	
	We now turn our attention to the map $\Xi_w$. The converse inclusion 
	\[\Xi_w\left(\KKoh(w)\right)\subseteq \SVRT(w)\]
	requires extracting recursive information from a given \K-Kohnert diagram, with no knowledge of any specific move history yielding it. Much of this work does not demand any pattern avoidance assumptions.
	
	The following lemma describes the effect of a single move on the threads in a column. Its proof is straightforward from the definitions. See Figure~\ref{fig:thread_change_under_single_move} for a proof by picture.
	\begin{lemma}
		\label{lem:column-word-update}
		Let $E=(B,G)$ be a grave diagram on $[n]\times [n]$ thread-compatible with $w$. Suppose an ordinary or \K-Kohnert move on $E$ moves a box from row $q$ to row $p<q$ in column $c$ of $E$, producing $E'$. Then $E'$ is also thread-compatible with $w$. Specifically, let the threads in column $c$ with leading box in any of the rows $p,p+1,p+2,\ldots,q$ have row sets 
		\[\{p+1\},\,\{p+2\},\,\ldots,\,\{q-1\},\,\{q\}\cup H \quad\mbox{for}\quad H\subseteq\{q+1,q+2,\ldots,n\}.\]
		Then the threads occupying rows $p,p+1,p+2,\ldots,q$ in $E'$ have row sets
		\[
		\{p\},\,\{p+1\},\,\ldots,\,\{q-2\},\,J\cup H
		\]
		where
		\[
			J=
			\begin{cases}
				\{q-1\} &\mbox{ for an ordinary Kohnert move,}\\
				\{q-1,q\} &\mbox{ for a \K-Kohnert move.}	
			\end{cases}
		\]
	\end{lemma}
	
	\begin{figure}[ht]
		\centering
		\includegraphics[scale=.85]{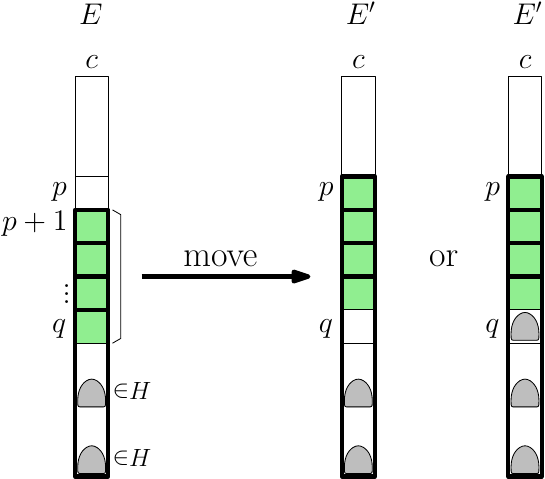}
		\caption{A visualization of Lemma~\ref{lem:column-word-update}.}
		\label{fig:thread_change_under_single_move}
	\end{figure}

	Since $D(w)$ is trivially thread-compatible with $w$, it follows from Lemma~\ref{lem:column-word-update} that all elements of $\KKoh(w)$ are thread-compatible with $w$. We next analyze the threads in a given column of any $E\in \KKoh(w)$.
	
	\begin{lemma}
		\label{lem:one-column}
		Fix a column $c$ of $E\in \KKoh(w)$, and let column $c$ of $D(w)$ have boxes in rows $i_1<i_2<\cdots<i_m$. Suppose that from north to south, the threads of column $c$ of $E$ have row sets $R_1,\ldots,R_m$. Then 
		\[\max(R_k)\leq i_k\quad\text{for all } 1\leq k\leq m.\]
	\end{lemma}
	
	\begin{proof}
		For $E=D(w)$, clearly $R_k = \{i_k\}$ and the result is immediate. Suppose that the result holds for $E$, and consider a diagram $E'\in\KKoh(w)$ obtained from a single move on $E$. Observe that any move shifting boxes in any other column will not affect the threads of column $c$, so assume $E'$ is obtained from a move shifting a box in column $c$ from row $q$ to row $p<q$. 
		
		The affected threads in column $c$ of $E$ are consecutive. Suppose they are 
		\[R_a,R_{a+1},\ldots,R_{a+M},\]
		so $M=q-p-1$. Then from the anatomy of a single move,
		\[
			R_{a+k}=\{p+1+k\}\quad\mbox{for}\quad(0\leq k<M), \qquad R_{a+M}=\{q\}\cup H
		\]
		for some set $H$ of ghost rows below $q$. Write $R'_1,\ldots,R'_m$ for the threads of column $c$ of $E'$. Then by Lemma~\ref{lem:column-word-update},
		\[
			R'_{a+k}=\{p+k\}\quad(0\leq k<M),
		\]
		with $R'_{a+M}$ being either $\{q-1\}\cup H$ or $\{q-1,q\}\cup H$ according to the type of move. Hence 
		\[\max(R'_k)\leq \max(R_k)\leq i_k \quad\mbox{for all } k.\qedhere\]
	\end{proof}
	
	\begin{corollary}
		\label{cor:column-threading}
		Let $w\in S_n$ and $E\in\KKoh(w)$. Then $\Xi_w(E)$ is defined, and satisfies both the column-strict condition \eqref{eq:col-cond} and the flag condition \eqref{eq:flag-cond}.
	\end{corollary}
	
	\begin{proof}
		Lemma~\ref{lem:column-word-update} shows that $E$ is thread-compatible with $w$. Lemma~\ref{lem:one-column} gives the flag condition. The column-strict condition holds trivially since all row indices in an earlier thread are strictly smaller than all row indices in the next thread.
	\end{proof}
	
	\subsection{The Row Condition and 1432-Avoidance}\phantom{}\newline\vspace{-2ex}
	
	To verify the inclusion $\Xi_w\left(\KKoh(w)\right)\subseteq \SVRT(w)$, it remains to prove the row condition~(\ref{eq:row-cond}) is satisfied by any $\Xi_w(E)$. No argument prior to this point has required $1432$-avoidance. The first place in this subsection that 1432-avoidance will be assumed is Proposition~\ref{prop:threading}. The following example illustrates the problem with 1432-patterns.
	
	\begin{example}
		\label{ex:1432}
		Let $w=1432$. Starting from $D(w)$, move the box at $(3,2)$ to $(1,2)$ by an ordinary
		Kohnert move, then apply $\Xi_w$:
		\begin{center}
			\includegraphics[scale=1]{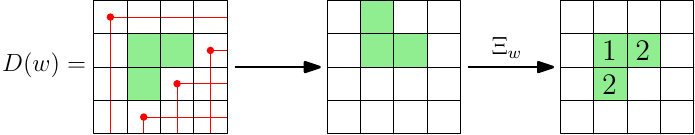}
		\end{center}
		The row condition~(\ref{eq:row-cond}) fails in row $2$.
	\end{example}
	
	We show in this subsection that such minimal failures of the map $\Xi_w$ to satisfy the row condition of SVRT are always caused locally by 1432-patterns in $w$. For the remainder of this subsection, we will study the following situation.
	
	\begin{setup}
		\label{set:1432_setup}
		Fix $E\in \KKoh(w)$ and a single ordinary or \K-Kohnert move $E\longrightarrow E'$ shifting a box in column $c$ from row $q$ to row $p<q$. Let $T=\Xi_w(E)$ and $T'=\Xi_w(E')$. 
	\end{setup}
	
	\begin{center}
		\begin{tikzcd}[row sep=large, column sep=huge]
			E\arrow[r,"\text{single\,move}"] \arrow[d, "\Xi_w"'] & E' \arrow[d, "\Xi_w"]\\
			T& T'
		\end{tikzcd}
	\end{center}
	
	Consider the execution of $\Xi_w(E)$. We first analyze the threads of $E$ and $E'$. Note that the threads in all columns except column $c$ are identical for $E$ and $E'$. Summarily, when passing from $T$ to $T'$, only entries of cells within column $c$ of $D(w)$ change. 
	
	We address column $c$ by rephrasing Lemma~\ref{lem:column-word-update} to this setting. 
	
	\begin{lemma}
		\label{lem:local-update}
		Assume Setup~\ref{set:1432_setup}.
		Suppose the threads in column $c$ of $E$ whose leading boxes lie in rows
		$p+1,p+2,\ldots,q$ are assigned to $(r_0,c),(r_1,c),\ldots,(r_M,c)\in D(w)$, where $M=q-p-1$ and
		$r_0<r_1<\cdots<r_M$.
		For all $0\leq j<M$,
		\[T(r_j,c)=\{p+1+j\},\quad\mbox{and}\quad T'(r_j,c)=\{p+j\}.\]
		For $j=M$, write $T(r_M,c)=\{q\}\cup H$ for $H\subseteq\{q+1,q+2,\ldots,n\}$. Then
		\[
			T'(r_M,c) = 
			\begin{cases}
				\{q-1\}\cup H &\mbox{if } E\longrightarrow E'\mbox{ is an ordinary Kohnert move,}\\
				\{q-1,q\}\cup H &\mbox{if } E\longrightarrow E'\mbox{ is a \K-Kohnert move.}
			\end{cases}
		\]
	\end{lemma}
	
	Next, we study the possibilities for a first-time failure of the SVRT row condition. Assume $T$ satisfies the row condition~(\ref{eq:row-cond}) and $T'$ does not. By Lemma~\ref{lem:local-update} and its preceding paragraph, $T$ and $T'$ differ only in the boxes $(r_j,c)$ for $j=0,\ldots,M$. Comparing their differing entries in column $c$ shows that every violation of (\ref{eq:row-cond}) in $T'$ occurs between column $c$ and a column $c'>c$. The next lemma documents the local situation in more detail. See Figure~\ref{fig:situation1} for a visual aid.
	
	\begin{figure}[ht]
		\centering
		\includegraphics[scale=.9]{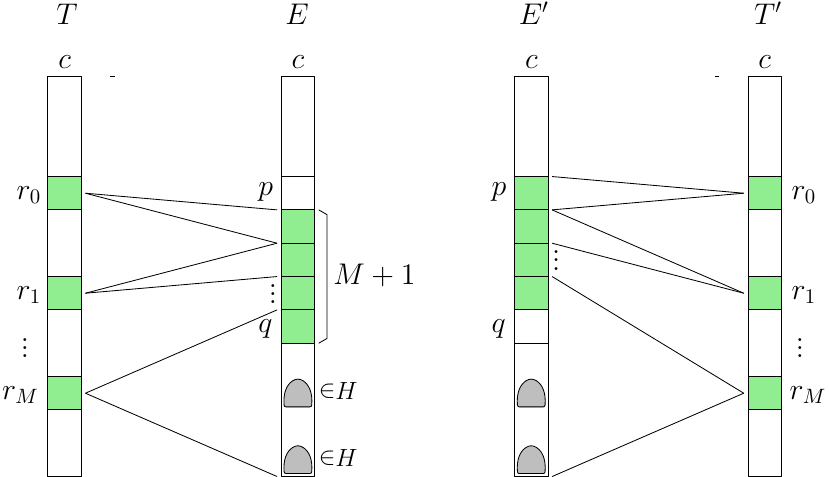}
		\caption{The situation of Lemma~\ref{lem:critical-equality}.}
		\label{fig:situation1}
	\end{figure}
	
	\begin{lemma}
		\label{lem:critical-equality}
		Assume Setup~\ref{set:1432_setup}.
		Suppose $T$ satisfies all row inequalities \textup{(\ref{eq:row-cond})} and $T'$ does not. Then there exists an index $k$ with $0\leq k<M$, and a column index $d>c$ such that $(r_k,d)\in D(w)$ satisfying:
		\begin{itemize}
			\item $\max \, T(r_k,d)=p+1+k$;
			\item $T(r_k,c)=\{p+1+k\}$; and
			\item $T'(r_k,c)=\{p+k\}$.
		\end{itemize}
	\end{lemma}
	
	\begin{proof}
		By Lemma~\ref{lem:local-update}, a new violation in $T'$ must have the form
		\[\max \, T(r_k,d) = \max \, T'(r_k,d) = \min \, T'(r_k,c) + 1.\]
		Some violation exists by assumption, and must occur in one of the rows $r_0,\ldots,r_M$. Suppose the violation occurs in row $r_k$. It remains to show that $k\neq M$. Assume for a contradiction that $k=M$.
		
		Recall $M=q-p-1$, so
		\[\min \, T'(r_M,c)=q-1 \quad\mbox{and}\quad \max \, T(r_M,d) = \max \, T'(r_M,d)=q.\]		
		But $E\longrightarrow E'$ via a move, so the box in row $q$ is rightmost in $E$. Then for any $d>c$, it follows that $q$ does not appear in column $d$ of $T$ (or therefore $T'$). This would contradict $\max \, T(r_k,d)=q$.
	\end{proof}
	
	\begin{remark}
		\label{rmk:simple_2}
		Consider the special case of the previous lemma when $q=p+1$. In this case, $M=q-p-1 = 0$, so the lemma has $0\leq k<0$. Consequently, the hypotheses of the lemma cannot be met. Hence $T'$ satisfies the row condition~(\ref{eq:row-cond}). We utilize this fact in Section~\ref{sec:simple}, where we study these \emph{simple} moves. See the proof of Lemma~\ref{lem:simple_xi_codomain}.
	\end{remark}
	
	\begin{lemma}
		Assume Setup~\ref{set:1432_setup}.
		If $T$ satisfies all row inequalities \textup{(\ref{eq:row-cond})} and $T'$ does not, then there exists a row index
		$a<r_0$ such that $w(a)<c$.
	\end{lemma}
	
	\begin{proof}
		Assume no such $a$ exists. For every $a<r_0$, the assumption gives $w(a)\geq c$. Since $(r_0,c)\in D(w)$, necessarily $w^{-1}(c)>r_0$. But note equality $w(a)=c$ is impossible since $w^{-1}(c)>r_0$. Hence $(a,c)\in D(w)$ for every $a<r_0$. Thus column $c$ of $D(w)$ has cells in all rows $1,2,\ldots,r_0-1$ above $(r_0,c)$.
		
		By Corollary~\ref{cor:column-threading}, the $r_0-1$ threads assigned to these higher cells are strictly above the thread assigned to $(r_0,c)$. In $T$, the thread assigned to $(r_0,c)$ has minimum $p+1$. Therefore those $r_0-1$
		higher threads each contain at least one distinct cell from the index set $\{1,2,\ldots,p\}$. Then $r_0-1\leq p$, i.e.~$r_0\leq p+1$. Conversely, the flag condition on $(r_0,c)$ gives $p+1\leq r_0$. Consequently $r_0=p+1$.
		
		But then those $p=r_0-1$ higher threads occupy every box $(1,c),\ldots,(r_0-1,c)$, with $T(i,\,c)=\{i\}$ for $1\leq i\leq p$. In particular, $T(r_0-1,c)=\{p\}$. However, recall from Setup~\ref{set:1432_setup} that the move giving $E\longrightarrow E'$ shifted a box in column $c$ from row $q$ to row $p$. In particular $(p,c)\notin E$. This contradicts $T(r_0-1,c)=\{p\}$.
	\end{proof}
	
	We now show the minimal failure of (\ref{eq:row-cond}) in $T'$ locally induces a 1432-pattern in $w$. We recall the Rothe diagram membership criterion
	\begin{align*}
		(i,j)\in D(w)\quad\Longleftrightarrow\quad
		i<w^{-1}(j)\text{ and }w(i)>j.
	\end{align*}
	Set $t=w^{-1}(c)$, the row index of the dot in column $c$ of $D(w)$. Then
	\[(r_0,c),\ldots,(r_M,c)\in D(w)\quad\mbox{implies}\quad t>r_M>\cdots>r_0.\]
	
	\begin{lemma}
		\label{lem:crazy_lemma}
		Assume Setup~\ref{set:1432_setup}.
		Suppose $T$ satisfies all row inequalities \textup{(\ref{eq:row-cond})} and $T'$ does not. Let $k$ and $d$ respectively be the indices guaranteed by Lemma \textup{\ref{lem:critical-equality}}. Then there exists a row index $s$ such that
		\[r_k<s<t=w^{-1}(c)\quad\mbox{and}\quad c<w(s)\leq d.\]
	\end{lemma}
	\begin{proof}
		We proceed by contradiction. Suppose no such $s$ exists. 
		
		\medskip
		
		\noindent \textbf{Step 1:} We first show that
		\begin{equation}
			\label{eq:no-sep-dot-order}
			w^{-1}(d)>t.
		\end{equation}
		Because $(r_k,d)\in D(w)$, it follows that $w^{-1}(d)>r_k$. Note the definition $t=w^{-1}(c)$ forces $w^{-1}(d)\neq t$. If it were the case that $w^{-1}(d)<t$, then $s=w^{-1}(d)$ would satisfy
		\[r_k<s<t \quad\mbox{and}\quad c<w(s)=d.\]
		This would contradict the assumed nonexistence of such an $s$. Hence (\ref{eq:no-sep-dot-order}) holds.
		
		\medskip
		
		\noindent \textbf{Step 2:} Next, we prove
		\[
			(r_j,d)\in D(w) \quad\text{for all $j$ with }k\leq j\leq M.
		\]
		For $j=k$, this is part of Lemma~\ref{lem:critical-equality}. Suppose $j>k$, so $r_k<r_j$. Since $(r_j,c)\in D(w)$, it follows that $r_j<t$ and $w(r_j)>c$. If it were the case that $c<w(r_j)\leq d$, then $s=r_j$ would satisfy
		\[r_k<s<t\quad\mbox{and}\quad	c<w(s)\leq d,\]
		contradicting the assumed nonexistence of such an $s$. Hence $w(r_j)>d$. 
		
		Recall that $t>r_j$, so by (\ref{eq:no-sep-dot-order}) one has $w^{-1}(d)>t>r_j$. But $w(r_j)>d$ and $r_j<w^{-1}(d)$ imply $(r_j,d)\in D(w)$. 
		
		\medskip
		
		\noindent \textbf{Step 3:} We next verify that 
		\begin{equation}
			\label{eq:no-sep-induction}
			\max \, T(r_k,d)=p+1+k,
			\quad\mbox{and}\quad		
			T(r_j,d) = \{p+1+j\} \quad\mbox{for } k<j\leq M.
		\end{equation}
		The case $j=k$ is covered by Lemma~\ref{lem:critical-equality}. We handle the values $k<j\leq M$ by finite induction.
		
		Assume $j>k$ and the assertion is known for $j-1$. Since $r_{j-1}<r_j$ and both cells $(r_{j-1},d),(r_j,d)\in D(w)$,
		the SVRT column-strictness condition (\ref{eq:col-cond}) applied to column $d$ of $T$ gives
		\[p+j=\max \, T(r_{j-1},d)<\min \, T(r_j,d).\]
		Conversely, the row condition~(\ref{eq:row-cond}) applied to row $r_j$ of $T$
		gives
		\[\max \, T(r_j,d)\leq \min \, T(r_j,c).\]
		But recall that by Lemma~\ref{lem:local-update},
		\[
			T(r_j,c) = \{p+1+j\} \mbox{ for } j<M,\quad\mbox{and}\quad
			\min \, T(r_M,c) = \min(\{q\}\cup H)=q=p+1+M.
		\]		
		Putting this all together, 
		\[p+j < \min \, T(r_j,d) \leq \max \, T(r_j,d)\leq p+1+j.\]
		By integrality then, $T(r_j,d) = \{p+1+j\}$ follows. 
		
		\medskip
		
		\noindent \textbf{Step 4:} We deduce a final contradiction to the original assumption that there does not exist an index $s$ satisfying the conclusion of the lemma. Since $M=q-p-1$, applying (\ref{eq:no-sep-induction}) with
		$j=M$ gives
		\[
		\max \, T(r_M,d)=p+1+M = q.
		\]
		But $q\in T(r_M,d)$ implies $E$ contains an occupied cell at $(q,d)$. However, recall the move $E\longrightarrow E'$ shifts a box at $(q,c)$ to $(p,c)$. Since $d>c$, an occupied cell in $(q,d)\in E$ would imply $(q,c)\in E$ was not rightmost, a contradiction. This contradiction proves the lemma.
	\end{proof}
	
	\begin{proposition}
		\label{prop:threading}
		If $w\in S_n$ is $1432$-avoiding, then $\Xi_w\left(\KKoh(w)\right)\subseteq \SVRT(w)$. 
	\end{proposition}
	
	\begin{proof}
		We work by contradiction. Fix $F\in\KKoh(w)$, and suppose that $\Xi_w(F)\notin \SVRT(w)$. By Corollary~\ref{cor:column-threading}, the noninclusion can only occur due to $\Xi_w(F)$ violating the row condition~(\ref{eq:row-cond}). Choose a sequence of ordinary/\K-Kohnert moves
		\[
		(D(w),\varnothing)=F_0\longrightarrow F_1\longrightarrow\cdots\longrightarrow F_N=F.
		\]
		Let $m$ be the first index for which $\Xi_w(F_m)$ violates the row condition. Set $E=F_{m-1}$ and $E'=F_{m}$. Suppose the move $E\longrightarrow E'$ shifts a box in column $c$ of $E$ from row $q$ to row $p<q$. Let $T=\Xi_w(E)$ and $T'=\Xi_w(E')$.
		
		This choice of notation puts us exactly in the setting of Setup~\ref{set:1432_setup}, and meets the assumptions of Lemmas~\ref{lem:local-update}--\ref{lem:crazy_lemma}.
		
		Applying these lemmas one-by-one provides indices $a,r_k,s,t$ satisfying
		\begin{itemize}
			\item $t=w^{-1}(c)$,
			\item $a<r_k$,
			\item $w(a)<c$,
			\item $r_k<s<t$, and
			\item $c<w(s)\leq d$.
		\end{itemize}
		Additionally, $(r_k,d)\in D(w)$ implies $w(r_k)>d$. Putting everything together:
		\[
			a<r_k<s<t\quad\mbox{and}\quad w(a)<c=w(t)<w(s)<w(r_k).
		\]
		Hence $w(a)w(r_k)w(s)w(t)$ forms a 1432-pattern in $w$. This contradiction proves the proposition.
	\end{proof}
	
	\subsection{The Bijection}\phantom{}\newline\vspace{-2ex}
	
	\begin{theorem}
		\label{thm:bijection}
		If $w$ is $1432$-avoiding, then $\Theta_w:\SVRT(w)\to \KKoh(w)$ is a weight-preserving bijection with inverse $\Xi_w$.
	\end{theorem}
	
	\begin{proof}
		That both maps are well-defined and have the appropriate codomain follows from Proposition~\ref{prop:theta-in-kkoh} and Proposition~\ref{prop:threading}. It was noted in (\ref{eq:theta-weight}) that $\Theta_w$ is weight-preserving. 
		
		That the two maps are inverses of each other follows easily from their columnwise definitions. 
		Starting with $T\in\SVRT(w)$, the entries of each set $T(r,c)=\{a_1<\cdots<a_m\}$ become exactly one thread in $\Theta_w(T)$ with row set $\{a_1,\ldots,a_m\}$. Hence $\Xi_w(\Theta_w(T))=T$. 
		
		Conversely, starting with $E\in\KKoh(w)$ and applying $\Xi_w$ records each thread as the filling of a box in $D(w)$. Applying $\Theta_w$ will then reconstruct the same leading box and the same following ghosts in each column, so $\Theta_w(\Xi_w(E))=E$.
	\end{proof}
	
	\begin{figure}[ht]
		\includegraphics[scale=.9]{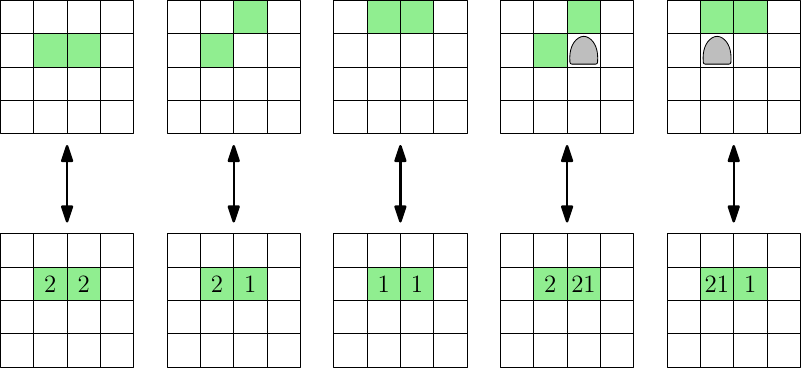}
		\caption{The bijection between \K-Kohnert diagrams and set-valued Rothe tableaux of $w=1423$.}
		\label{fig:1423_bijection}
	\end{figure}
	
	\begin{example}
		For the (1432-avoiding) permutation $w=1423$, the bijection $\KKoh(w)\leftrightarrow \SVRT(w)$ is shown in Figure~\ref{fig:1423_bijection}.
	\end{example}
	
	\begin{corollary}
		\label{cor:gkk_kk_equiv_1432}
		If $w\in S_n$ is 1432-avoiding, then $\KKoh(w)=\GKKoh(w)$. 
	\end{corollary}
	
	\begin{proof}
		By definition, $\KKoh(w)\subseteq \GKKoh(w)$. To verify the other inclusion, it is enough to show that if $E\in \KKoh(w)$ and $E\longrightarrow E'$ via a ghost Kohnert or ghost \K-Kohnert move, then $E'\in \KKoh(w)$.

		Suppose $E=(B,G)\in \KKoh(w)$ and $E'=(B',G')$ is a grave diagram with $E\longrightarrow E'$ via a ghost Kohnert or ghost \K-Kohnert move shifting a ghost in column $c$ from row $q$ to row $p<q$. It is immediate that column $c$ of $E'$ has the same number of threads as column $c$ of $E$, with each thread having the same leading box in the same row. Consequently, thread compatibility of $E'$ with $w$ follows from that of $E$ (which itself comes from Lemma~\ref{lem:column-word-update}). Let $T' = \Xi_w(E')$ be the corresponding set-valued filling of $D(w)$. Set $T=\Xi_w(E)\in \SVRT(w)$. We verify that $T'\in\SVRT(w)$ as well. 
		
		Suppose $(p_0,c)\in B$ leads the thread containing $(p+1,c),(p+2,c),\ldots,(q,c)\in G$.
		Let $(r,c)\in D(w)$ be the Rothe cell assigned to this affected thread.
		Set
		\begin{align*}
			H_0&=\{t\mid p_0<t<p \text{ and } (t,c)\in G\},\\
			H_1&=\{t\mid q<t,\ (t,c)\in G,\text{ and }(t,c) \text{ lies in the same thread of }E\text{ as }(p_0,c)\}.
		\end{align*}
		We provide an illustration of the situation in Figure~\ref{fig:thread_change_under_single_ghost_move}. 
		
		\begin{figure}[ht]
			\centering
			\includegraphics[scale=.85]{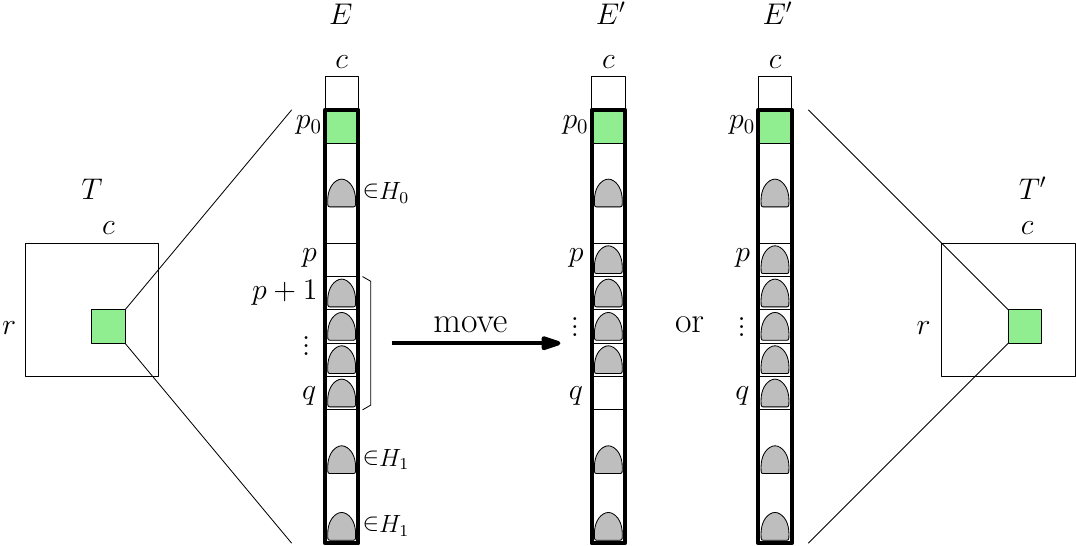}
			\caption{The situation in the proof of Corollary~\ref{cor:gkk_kk_equiv_1432}.}
			\label{fig:thread_change_under_single_ghost_move}
		\end{figure}
		
		Let 
		\[S=\{p_0\}\cup H_0\cup \{p+1,p+2,\ldots,q\}\cup H_1.\]
		Suppose $(r,c)\in D(w)$ is the Rothe cell with $T(r,c)=S$. Then
		\[
			T'(r,c) = 
			\begin{cases}
				(S\setminus \{q\}) \cup \{p\} &\mbox{for $E\longrightarrow E'$ via a ghost Kohnert move,}\\
				S \cup \{p\}&\mbox{for $E\longrightarrow E'$ via a ghost \K-Kohnert move.}
			\end{cases}
		\]
		It is immediate that
		\[\min\, T'(r,c) = p_0 = \min\, T(r,c) \quad\mbox{and}\quad \max\,T'(r,c)\leq \max\,T(r,c).\]
		Note $T$ and $T'$ agree on all other Rothe cells. The column condition (\ref{eq:col-cond}), flag condition (\ref{eq:flag-cond}), and row condition~(\ref{eq:row-cond}) for $T'$ follow from those for $T$. Consequently, $T'\in \SVRT(w)$, so $E'=\Theta_w(T')\in \KKoh(w)$.	
	\end{proof}
	
	Notice that the previous argument did not require the rightmostness of the source ghost in a ghost move. Consequently, it proves a slightly stronger statement analogous to \cite[Remark~5.9]{ghost_kohnert_fix}.
	
	\begin{corollary}
		\label{cor:1432-ry}
		Conjecture~\ref{conj:ry_kohnert} and Conjecture~\ref{conj:cr_kohnert} hold when $w$ is 1432-avoiding.
	\end{corollary}
	
	\begin{proof}
		By Fan--Guo's formula (Theorem~\ref{thm:fg_formula}),
		\[
		\mathfrak{G}_w(x)=\sum_{T\in\SVRT(w)}(-1)^{|\wt(T)|-\ell(w)}x^{\wt(T)}.
		\]
		Apply the bijection from Theorem~\ref{thm:bijection} and the equality from Corollary~\ref{cor:gkk_kk_equiv_1432}.
	\end{proof}
	
	\FloatBarrier
	
	\section{Simple \texorpdfstring{\K}{K}-Kohnert Moves}
	\label{sec:simple}
	
	Analogously to \cite[Theorem 4.1]{gao_prin_spec}, we characterize the permutations for which $\KKoh(w)$ can be generated without boxes ever moving up more than one row at a time.
	
	
	\begin{definition}
		Call an ordinary Kohnert move or a \K-Kohnert move \newword{simple} if it shifts a box upward by exactly one row. Let $\SKoh(w)$ be the closure of $D(w)$ under simple (ordinary) Kohnert moves. Let $\SKKoh(w)$ be the closure of $(D(w),\varnothing)$ under simple ordinary and simple \K-Kohnert moves.
	\end{definition}
	
	We show that for any (not necessarily 1432-avoiding) $w\in S_n$, 
	\[
	\begin{tikzcd}
		\SVRT(w) \arrow[r, bend left, "\Theta_w"] & \SKKoh(w) \arrow[l, bend left, "\Xi_w"]
	\end{tikzcd}
	\]
	are weight-preserving bijections. This will imply $\SKKoh(w)=\KKoh(w)$ for 1432-avoiding permutations.
	
	\begin{lemma}
		\label{lem:simple_theta_codomain}
		For any $w\in S_n$ and $T\in \SVRT(w)$, one has $\Theta_w(T)\in \SKKoh(w)$.
	\end{lemma}
	
	\begin{proof}
		The assertion $\Theta_w(T)\in \KKoh(w)$ is Proposition~\ref{prop:theta-in-kkoh}. That $\Theta_w(T)\in \SKKoh(w)$ follows from its proof, as noted in Remark~\ref{rmk:simple_1}.
	\end{proof}
	
	\begin{lemma}
		\label{lem:simple_xi_codomain}
		For any $w\in S_n$ and $E\in \SKKoh(w)$, one has $\Xi_w(E)\in \SVRT(w)$.
	\end{lemma}
	
	\begin{proof}
		Let $E\in\SKKoh(w)$, and choose a sequence of simple ordinary/K-Kohnert moves
		\[
			(D(w),\varnothing)=E_0\longrightarrow E_1\longrightarrow\cdots\longrightarrow E_N=E.
		\]
		Recall Corollary~\ref{cor:column-threading} shows that each $\Xi_w(E_m)$ is a well-defined set-valued filling of $D(w)$ satisfying both the column condition (\ref{eq:col-cond}) and the flag condition (\ref{eq:flag-cond}). To conclude $\Xi_w(E_m)\in\SVRT(w)$ for each $m$, it remains to check the row condition~(\ref{eq:row-cond}). Recall Remark~\ref{rmk:simple_2} notes that the row condition on $T'$ is implied by Lemma~\ref{lem:critical-equality} in the case of a simple move $E_m\longrightarrow E_{m+1}$.
	\end{proof}

	\begin{proposition}
		\label{prop:simple-svrt-bijection}
		For any permutation $w$, the map $\Theta_w$ is a weight-preserving bijection $\SVRT(w)\to \SKKoh(w)$ with inverse $\Xi_w$.
	\end{proposition}
	\begin{proof}
		Immediate from Lemma~\ref{lem:simple_theta_codomain} and Lemma~\ref{lem:simple_xi_codomain}.
	\end{proof}
	
	Denote the \newword{single-valued Rothe tableaux} of $w$ by
	\[\SRT(w) = \{T\in\SVRT(w)\mid |\wt(T)|=\ell(w)\}.\]
	
	\begin{corollary}
		\label{cor:simple-srt-bijection}
		For every permutation $w\in S_n$, the restriction of $\Theta_w$ gives a bijection $\SRT(w)\to\SKoh(w)$.
	\end{corollary}
	
	We recall another result of Fan--Guo.
	
	\begin{theorem}[{\cite[Theorem 2.13]{fg_svrt}}]
		\label{thm:fg_1432_schub}
		If $w$ contains a 1432 pattern, then
		\[\mathfrak{S}_w\neq \sum_{T\in \SRT(w)} x^{\wt(T)}.\]
	\end{theorem}
	
	\begin{theorem}
		\label{thm:1432_kohnert_characterization}
		For $w\in S_n$, the following are equivalent:
		\begin{enumerate}[label=\textup{(\roman*)}]
			\item $w$ is 1432-avoiding;
			\item $\SKKoh(w)=\KKoh(w)$.
		\end{enumerate}
	\end{theorem}
	
	\begin{proof}
		Suppose first that \(w\) is 1432-avoiding. By Proposition~\ref{prop:simple-svrt-bijection} and Theorem~\ref{thm:bijection},
		\[
		\SKKoh(w)=\Theta_w(\SVRT(w)) = \KKoh(w).
		\]
		
		Conversely, suppose $w$ contains a 1432-pattern. By Kohnert's rule (Theorem~\ref{thm:kohnert_rule}),
		\[\mathfrak{S}_w=\sum_{E\in \Koh(w)} x^{\wt(E)}.\]
		But Theorem~\ref{thm:fg_1432_schub} and Corollary~\ref{cor:simple-srt-bijection} imply 
		\[\mathfrak{S}_w \neq \sum_{T\in \SRT(w)} x^{\wt(T)} = \sum_{E\in \SKoh(w)} x^{\wt(E)}.\]
		Hence there exists $E\in \Koh(w)\setminus \SKoh(w)$, so in particular $\SKKoh(w)\neq \KKoh(w)$.
	\end{proof}

	\FloatBarrier
	\clearpage
	
	\appendix
	
	\section{Diagram Justifications}
	\label{sec:appendix}
	
	For completeness, we use this section to provide explicit constructions of various ghost Kohnert diagrams referred to in Section~\ref{sec:conjectures}. For compactness, we freely omit bottommost empty rows and any empty columns from diagrams. We label a move $E\longrightarrow E'$ by the row index $i$ where the move occurs, adding a hat like $\widehat{i}$ to indicate a \K-Kohnert or ghost \K-Kohnert move.
	
	\begin{figure}[ht]
		\begin{center}
			\includegraphics[scale=1]{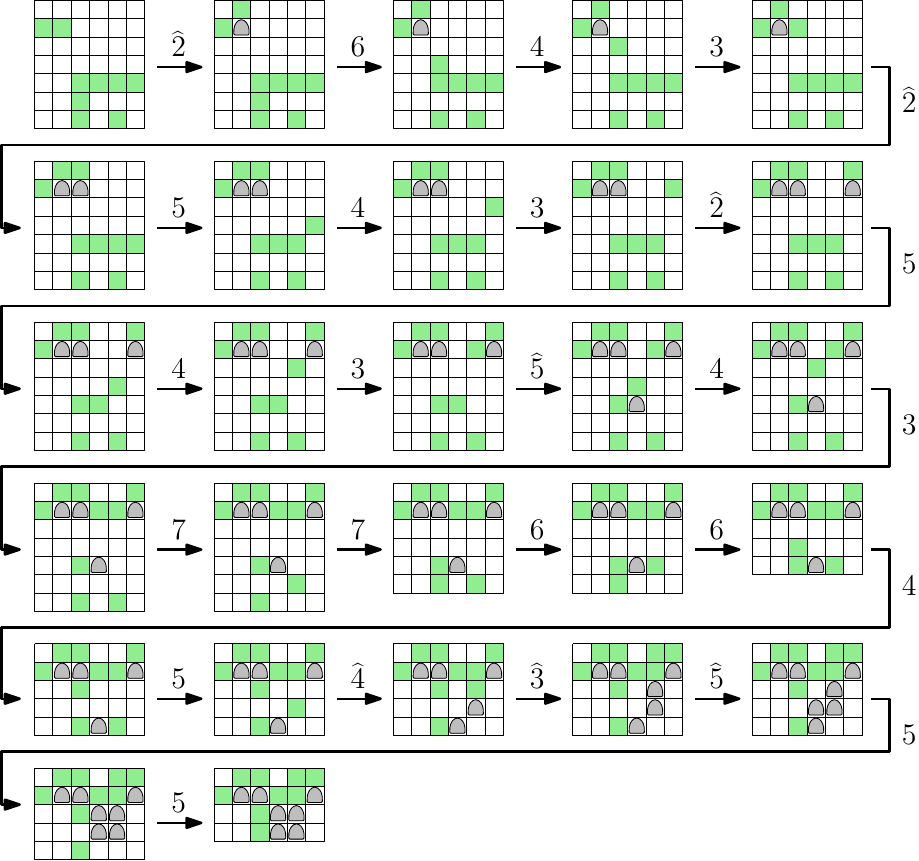}
			\caption{Achieving the left diagram of Figure~\ref{fig:142396857_4633_ghost_koh_overcounts}.}
			\label{fig:142396857_4633_ghost_koh_1_steps}
		\end{center}
	\end{figure}
	
	\begin{figure}[ht]
		\begin{center}
			\includegraphics[scale=1]{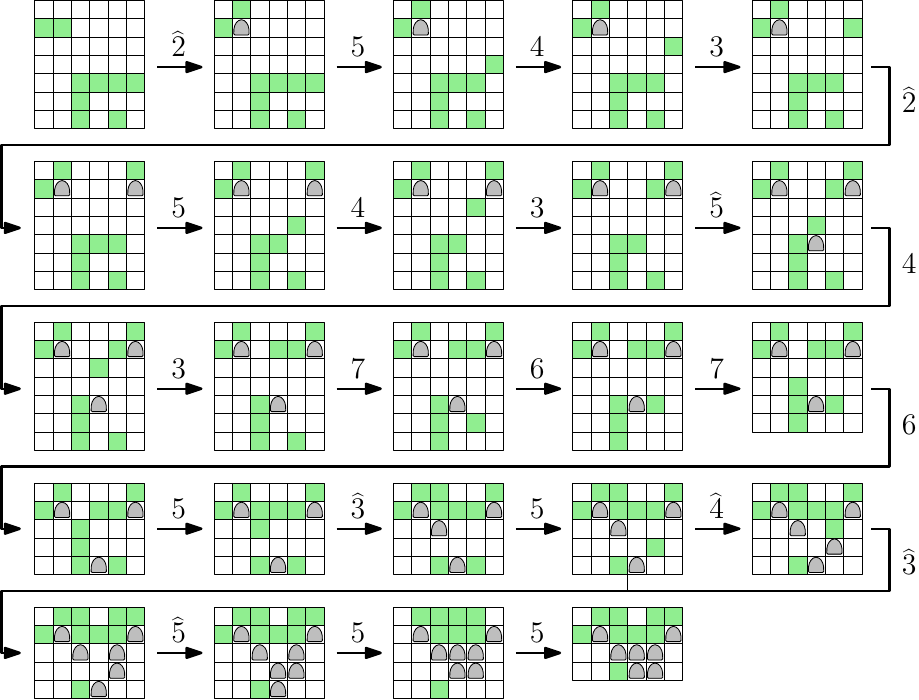}
			\caption{Achieving the right diagram of Figure~\ref{fig:142396857_4633_ghost_koh_overcounts}.}
			\label{fig:142396857_4633_ghost_koh_2_steps}
		\end{center}
	\end{figure}
	
	\begin{figure}[ht]
		\begin{center}
			\includegraphics[scale=1]{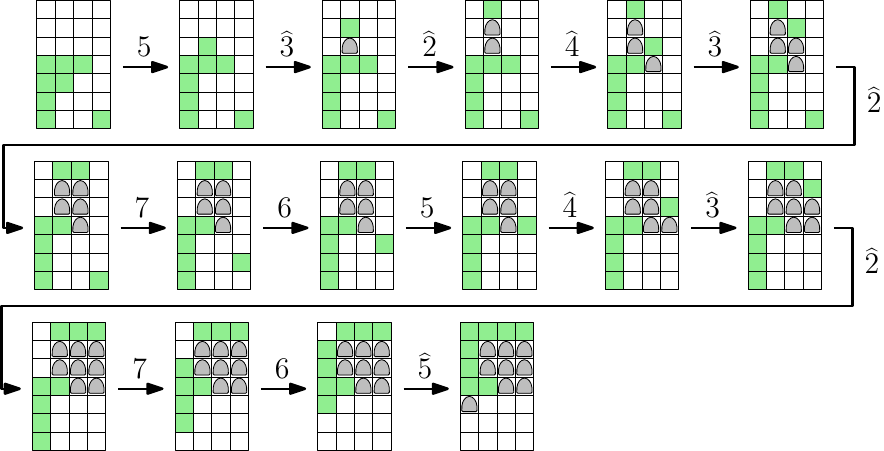}
			\caption{Achieving the left diagram of Figure~\ref{fig:123765948_44441_ghost_koh_undercounts}.}
			\label{fig:123765948_44441_ghost_koh_1_steps}
		\end{center}
	\end{figure}
		
	\begin{figure}[ht]
		\begin{center}
			\includegraphics[scale=1]{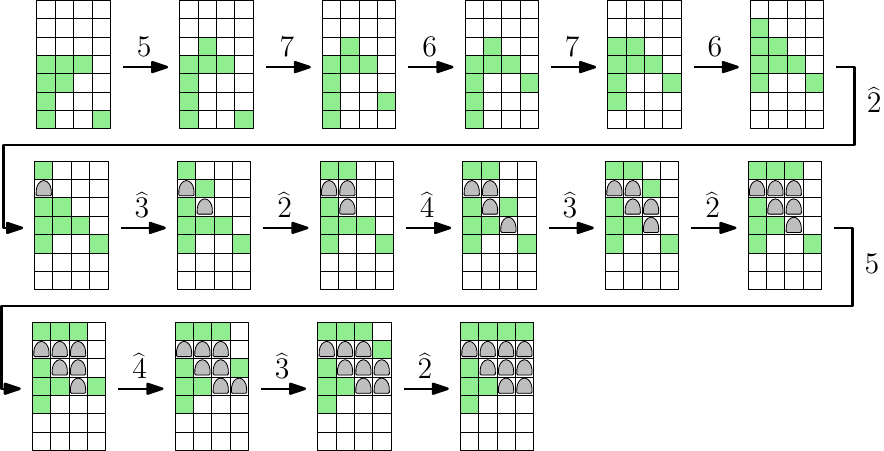}
			\caption{Achieving the right diagram of Figure~\ref{fig:123765948_44441_ghost_koh_undercounts}.}
			\label{fig:123765948_44441_ghost_koh_2_steps}
		\end{center}
	\end{figure}
	
	\begin{figure}[ht]
		\begin{center}
			\includegraphics[scale=1]{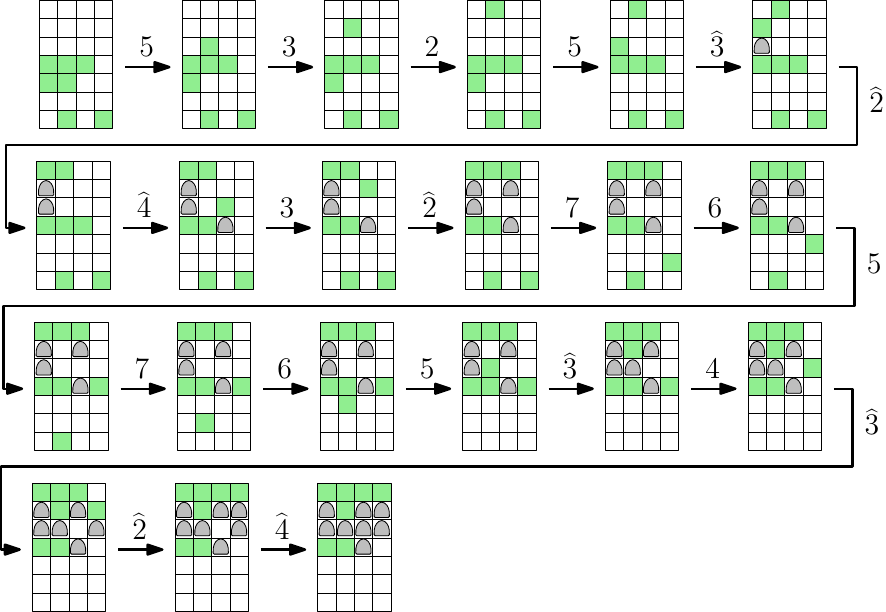}
			\caption{Achieving the diagram in Figure~\ref{fig:ry_support_correction}.}
			\label{fig:123764958_4443_support_ghost_koh}
		\end{center}
	\end{figure}
	
	\FloatBarrier
	
	\section*{Acknowledgements}
	
	The author is grateful to Colleen Robichaux for her insightful feedback on this paper, and for helpful conversations about \K-Kohnert rules. The author thanks Sam Castonguay for her support. The author also thanks Dave Anderson for his work developing the Julia package SchubertPolynomials.jl, which the author used in many computations referenced in this paper.
	
	\bibliographystyle{plain}
	\bibliography{k_kohnert_bibliography}
\end{document}